\documentclass{amsart}

\usepackage{caption, amssymb,amsfonts,amsmath,amsthm,graphicx}
\usepackage[extension=pdf,bookmarksopen]{hyperref}
\usepackage{mathtools}
\usepackage[nice]{nicefrac}
\usepackage[dvipsnames]{xcolor}
\usepackage{mathrsfs}
\usepackage{pgf,tikz}
\usepackage[totalheight=9.6in,totalwidth=6.15in,marginparwidth=1in,centering]{geometry}
\usetikzlibrary{calc}
\usepackage{relsize}
\usepackage{enumitem}
\usepackage[normalem]{ulem}
\usepackage{dsfont}
\mathtoolsset{showonlyrefs=true}

\definecolor{darkblue}{rgb}{0.13,0.13,0.39}
\hypersetup{colorlinks=true,urlcolor=darkblue,citecolor=darkblue,linkcolor=darkblue,%
  pdftitle=,pdfauthor={A. Linker, D. Remenik}}

\usetikzlibrary{patterns}
\usetikzlibrary{arrows}    
\usetikzlibrary{calc}
\usepackage{relsize}
\usetikzlibrary{patterns}

\tikzset{fontscale/.style = {font=\relsize{#1}}
    }

\newcommand{\one}{{\mathsf 1}}

\newcommand{\R}{\mathbb R}

\newcommand{\1}{\bf{1}}

\newcommand{\E}{\mathbb E}
\renewcommand{\P}{\mathbb P}

\newcommand{\eL}{\eta^{\text{w}}}
\newcommand{\ec}{\eta^{p}}

\newcommand{\uno}[1]{\mathbf{1}_{#1}}
\newcommand{\cad}{c\`adl\`ag}
\newcommand{\G}{G}
\newcommand{\clambda}{\bar{\lambda}_{\G}}

\renewcommand{\phi}{\varphi}

\newcommand{\Z}{\mathbb Z}
\newcommand{\N}{\mathbb N}

\renewcommand{\P}{\mathbb P}

\newcommand{\x}{\mathbf{x}}

\newlength{\hatchspread}
\newlength{\hatchthickness}
\newlength{\hatchshift}
\newcommand{\hatchcolor}{}
\tikzset{hatchspread/.code={\setlength{\hatchspread}{#1}},
         hatchthickness/.code={\setlength{\hatchthickness}{#1}},
         hatchshift/.code={\setlength{\hatchshift}{#1}},
         hatchcolor/.code={\renewcommand{\hatchcolor}{#1}}}
\tikzset{hatchspread=3pt,
         hatchthickness=0.4pt,
         hatchshift=0pt,
         hatchcolor=black}
\pgfdeclarepatternformonly[\hatchspread,\hatchthickness,\hatchshift,\hatchcolor]
   {custom north west lines}
   {\pgfqpoint{\dimexpr-2\hatchthickness}{\dimexpr-2\hatchthickness}}
   {\pgfqpoint{\dimexpr\hatchspread+2\hatchthickness}{\dimexpr\hatchspread+2\hatchthickness}}
   {\pgfqpoint{\dimexpr\hatchspread}{\dimexpr\hatchspread}}
   {
    \pgfsetlinewidth{\hatchthickness}
    \pgfpathmoveto{\pgfqpoint{0pt}{\dimexpr\hatchspread+\hatchshift}}
    \pgfpathlineto{\pgfqpoint{\dimexpr\hatchspread+0.15pt+\hatchshift}{-0.15pt}}
    \ifdim \hatchshift > 0pt
      \pgfpathmoveto{\pgfqpoint{0pt}{\hatchshift}}
      \pgfpathlineto{\pgfqpoint{\dimexpr0.15pt+\hatchshift}{-0.15pt}}
    \fi
    \pgfsetstrokecolor{\hatchcolor}
    \pgfusepath{stroke}
   }

\def \mI {\mathcal{I}}
\def \mIA {\mathcal{I}_A}
\def \mIR {\mathcal{I}_R}
\def \mU {\mathcal{U}}
\def \mbI { \bar{\mathcal{I}}}
\def \mR {\mathcal{R}}
\def \mO { \mathcal{O}}
\def \mC { \mathcal{C}}

\newcommand{\lL}{\xrightarrow{\text{w}}}

\theoremstyle{plain}
\newtheorem{theorem}{Theorem}[section]
\newtheorem{corollary}[theorem]{Corollary}
\newtheorem{proposition}[theorem]{Proposition}
\newtheorem{lemma}[theorem]{Lemma}

\theoremstyle{definition}

\newtheorem{remark}{Remark}[section]
\newtheorem{defn}[theorem]{Definition}

\numberwithin{equation}{section}

\begin{document}

\title{The contact process with dynamic edges on $\Z$}

\author{Amitai Linker} \address[A.~Linker]{
  Depto. de Ingenier\'ia Matem\'atica\\
  Universidad de Chile\\
  Av. Beauchef 851, Torre Norte, Piso 5\\
  Santiago\\
  Chile} \email{alinker@dim.uchile.cl}

\author{Daniel Remenik} \address[D.~Remenik]{
  Depto. Ingenier\'ia Matem\'atica and Centro de Modelamiento Matem\'atico (UMI-CNRS 2807)\\
  Universidad de Chile\\
  Av. Beauchef 851, Torre Norte, Piso 5\\
  Santiago\\
  Chile} \email{dremenik@dim.uchile.cl}

\maketitle

\begin{abstract}

  We study the contact process running in the one-dimensional lattice undergoing dynamical percolation, where edges open at rate $vp$ and close at rate $v(1-p)$. 
  Our goal is to explore how the speed of the environment, $v$, affects the behavior of the process.
  Among our main results we find that:
  1. For small enough $v$ the process dies out, while for large $v$ the process behaves like a contact process on $\Z$ with rate $\lambda p$, where $\lambda$ is the birth rate of each particle, so in particular it survives if $\lambda$ is large.
  2. For fixed $v$ and small enough $p$ the network becomes immune, in the sense that the process dies out for any infection rate $\lambda$, while if $p$ is sufficiently close to $1$ then for all $v>0$ survival is possible for large enough $\lambda$.
  3. Even though the first two points suggest that larger values of $v$ favor survival, this is not necessarily the case for small $v$: when the number of initially infected sites is large enough, the infection survives for a larger expected time in a static environment than in the case of $v$ positive but small.
  Some of these results hold also in the setting of general (infinite) vertex-transitive regular graphs.
\end{abstract}

\section{Introduction}\label{introduccion}

Since it was first introduced by Harris in \cite{harris1974} more than forty years ago, the contact process has turned into one of the most widely used models for population growth.
While most of the early work was done for the process on the Euclidean lattice, much of the interest in more recent years has focused on studying the contact process running on random graphs, in an attempt to understand procceses of this type in settings which capture in a better way the main features of real-world networks, whether technological, social, economic or biological in nature; this has lead to tremendous progress in our understanding of both the contact process and some random graph models (see \cite{Durrett2,lalley2017,can_2018} to name a few).
However, for the most part this work has taken place in the context of static random networks; real-world networks, on the other hand, tend to be dynamic in nature, a characteristic that might have a large impact in the qualitative behavior of the process.
In the mathematical literature, this impact has received relatively little attention (see \cite{broman2007,broman2006,Remenik}, and more recently \cite{jacobmorters,metastable}, for some contributions in this direction).

The goal of this work is to study the contact process running on a very simple dynamic environment which captures one of the most important features of dynamical networks: the continuous merging and division of connected components.
This feature has been studied for related processes such as the SIR disease model or the PUSH-PULL rumor spreading protocol in some recent works (see \cite{Durarxiv,ball,inpro,inpro2}).
However, in all of those cases the choice of dynamics for the network tends to reduce the number of connections in the underlying graphs, and thus it works against the growth of the process.
Our interest in this paper, in contrast, is to understand the effect that the introduction of dynamics has on the process independently of changes in the topology of the network.
To this end we study the contact process on networks undergoing a simple stationary dynamics for the environment in which its edges alternate between \textit{available} and \textit{unavailable} independently from each other at some given speed $v$, and focus on this parameter as a measure of the rate of change.
Since this speed increases both the rate of connection and of disconnection, it is not clear a priori whether this hurts or helps the spread of the infection.
We will show that, as opposed to what was observed in \cite{metastable}, in broad terms making the speed of the environment large turns out to favor survival.

We turn now to the definition of the model and our main results.
We will present our results in the simplest case when the base graph is the one-dimensional lattice, as a toy model where we are able to distinguish clearly between the different behaviors that the system presents for $v$ small and $v$ large.
However, with a bit of work most of our proofs (the main exception being Theorem \ref{teo1}, concerning small $v$) are valid in the setting of general vertex-transitive regular graphs.
We discuss briefly this extension in Section \ref{sec:genG}, and then present the corresponding proofs in Section \ref{proofs} in this more general setting.

\section{Setting and results}
 \label{Settings}

 \subsection{The CPDE on $\Z$}\label{sec:CPDEonZ}
 
Consider the one-dimensional lattice $(\Z,E)$ with $E$ the set of edges of the form $\{x,x+1\}$. The \textit{Contact Process with Dynamic Edges} $\{(\eta_t,\zeta_t)\}_{t\geq0}$ (from now on abbreviated CPDE) on $\Z$ is an interacting particle system made out of two processes, an \textit{environment} $\zeta_t\!:E\to\{0,1\}$ and an \textit{infection process} $\eta_t\!:\Z\to\{0,1\}$, whose transition rates we define locally as follows: for some fixed $v>0$ and $p\in(0,1)$ the environment evolves at any given $e\in E$ according to
 \begin{align*}
 0\longrightarrow 1&\quad \mbox{ at rate }\quad vp\\
 1\longrightarrow0&\quad \mbox{ at rate }\quad v(1-p),
 \end{align*}
 while for some fixed $\lambda>0$, the infection evolves at any given $x\in\Z$ according to
 \begin{align*}
 0\longrightarrow 1&\quad\mbox{ with rate }\quad\lambda n_{t,x}\\
 1\longrightarrow0&\quad\mbox{ with rate }\quad1,
 \end{align*}
 where $n_{t,x}=\eta_t(x-1)\zeta_t(\{x,x-1\})+\eta_t(x+1)\zeta_t(\{x,x+1\})$. We interpret $\zeta_t(e)=1$ as the edge $e$ being \textit{available} at time $t$, so that $n_{t,x}$ is the number of nearest neighbors of $x$ that are infected and connected to $x$ at that given moment. It can be checked that the rates above uniquely define the CPDE as a \cad~ Feller process in $\{0,1\}^{\Z\cup E}$, whose law we denote by $\P$ (or $\P_{\lambda,v,p}$ when we need to emphasize the dependence on the parameters). Even further, this process is monotone with respect to its initial condition, meaning that for any $a\in\{0,1\}^{\Z}$ and $b\in\{0,1\}^{E}$, the function $\P(\eta_t\geq a,\,\zeta_t\geq b\,|\eta_0,\zeta_0)$ is increasing on $\eta_0$ and $\zeta_0$ (where we use the usual partial pointwise order between $\{0,1\}$-valued functions). The proof of these facts is standard (see \cite{IPS1}). 
 
 The CPDE can be alternatively constructed by first sampling $\zeta$ and then running $\eta_t$ on the time-inhomogeneous graph defined by the environment (the so-called \textit{quenched} process). The environment process evolves as dynamical percolation on $(\Z,E)$ (as introduced in \cite{dynamicperc}); in our parametrization we introduce the parameter $v$, the \textit{environment speed} (i.e. the rate at which every edge updates its state), whose role in the behavior of the process we are interested in understanding. This process is stationary with respect to the product Bernoulli measure $\{0,1\}^E$ with density $p$, and in what follows we will assume (unless otherwise stated) that $\zeta_0$ is chosen at random using this distribution. This assumption allows us to identify $p$ as the density of available edges at any given time and, moreover, it allows us to attribute any effect of the evolution of the environment on the quenched process to its dynamics (and in particular its speed) rather than to changes in the properties of the network.
 
 From the form of the transition rates, for any fixed realization $\zeta$ of the environment the quenched infection process is a version of the contact process running on the evolving graph defined by $\zeta$. It follows, in particular, that the survival probability $\P_{\lambda,v,p}(\eta_t\neq0\;\forall t>0\,|\,\zeta)$ is increasing in $\lambda$. Averaging with respect to $\zeta$ we deduce that the annealed survival probability $\P_{\lambda,v,p}(\eta_t\neq0\;\forall t>0)$ satisfies the same property, so it makes sense to define a critical parameter $\lambda_0(v,p)$ for survival of the CPDE:
 \begin{equation*}
 \lambda_0(v,p)\;=\;\inf\!\big\{\lambda>0,\;\P_{\lambda,v,p}\big(\eta_t\neq0\;\forall t>0\big)>0\big\},
 \end{equation*}
 where we choose $\eta_0=\one_{\{0\}}$ as the initial condition for the infection process (it can be checked using standard arguments that $\lambda_0(v,p)$ is the same for any initial condition which contains a positive but finite number of infected sites).

\subsection{Main results}\label{sec:results}

 Our main goal in this paper is to give a (partial) description of the qualitative behaviour of $\lambda_0$ as a function of $v$ and $p$. An obvious first property is that $\lambda_0(p,v)$ is decreasing in $p$, since a higher density of open edges makes it easier for the infection to survive (this can be proved using a standard coupling argument). The dependence of $\lambda_0$ on $v$, on the other hand, is much subtler: increasing $v$ implies both opening and closing edges at a quicker pace, and it is not at all clear which of the two effects has a stronger impact on $\lambda_0$ (see also Remark \ref{rem:v}). But the following weaker version of monotonicity holds:
 
 \begin{proposition}
 	\label{monotoneprop}
 	The function $\frac{1}{v}\lambda_0(v,p)$ is non-increasing in $v$ for every value of $p$.
 \end{proposition}

 The following result provides some loose bounds on $\lambda_0(v,p)$, in terms of the critical parameter $\bar\lambda\coloneqq\lambda_0(0,1)$ of the standard contact process on $\Z$, which will be useful later on:
 
 \begin{proposition}
 	\label{lambda1}
 	For each $v,\lambda \geq 0$ and $p\in[0,1]$, the infection process in the CPDE is stochastically dominated from above by a contact process on $\Z$ with infection rate $\lambda$, and from below by a contact process on $\Z$ with infection rate $\beta(\lambda,v,p)$, where
 	$$\beta(\lambda,v,p)\;=\;\tfrac{1}{2}\big(\lambda+v-\sqrt{(v+\lambda)^2-4\lambda vp}\big).$$
 	As a consequence,
 	\[\bar{\lambda}\;\leq\;\lambda_0(v,p)\;\leq\;\hat\lambda(v,p)\]
 	where $\bar{\lambda}$ is the critical parameter of the contact process on $\Z$ and
 	\[\hat\lambda(v,p)=\bar{\lambda}\!\left(\frac{v-\bar{\lambda}}{vp-\bar{\lambda}}\right)\quad\text{if}\quad vp>\bar{\lambda},\qquad \hat\lambda(v,p)=\infty\quad\text{otherwise}.\]
 \end{proposition}
 
 The proof of the domination from below by a contact process with infection rate $\beta(\lambda,v,p)$ is based on a result of Broman \cite{broman2007}, while the formula for $\hat\lambda$ comes simply from solving $\bar\lambda\leq \beta(\lambda,v,p)$ for $\lambda$. 
 The expression for $\hat\lambda$ may seem opaque at first sight, but notice that
 \begin{equation}
 \label{eq:lc}
 \limsup_{v\rightarrow\infty}\lambda_0(v,p)\leq\lim_{v\rightarrow\infty}\hat\lambda(v,p)\;=\;\bar{\lambda}/p,
 \end{equation}
 which means that for any $\lambda>\bar{\lambda}/p$ the infection process survives if $v$ is large enough. 
 This proves the easier half of the first of our main theorems about $\lambda_0(v,p)$:
 
 \begin{theorem}
 	\label{teofast}
 	For any $p\in(0,1]$, \;$\lim_{v\rightarrow\infty}\lambda_0(v,p)\;=\;\bar\lambda/p$.
 \end{theorem}
 
 The idea is that if $v$ is large, the states of an edge at different times are almost independent, so we can approximate $\eta$ by a contact process with intensity $\lambda$ where each infection event is kept (independently) with probability $p$ and dismissed otherwise, which is simply a contact process on $\Z$ with intensity $\lambda p$, and hence we must have $\lambda_0(v,p)p\approx\bar\lambda$.
 
 Theorem \ref{teofast} together with \eqref{eq:lc} show that the upper bound $\lambda_0\leq\hat\lambda$ becomes sharp as $v$ gets large, but in general we expect $\lambda_0$ to be smaller than $\hat\lambda$.
 In particular, if we take any fixed $v\leq\bar{\lambda}$ then $\hat\lambda(v,p)=\infty$ for all $p$, but if $p\sim1$ then the infection hardly ever sees any closed edges, so we actually expect $\lambda_0\sim\bar{\lambda}$.
 
 On the other hand, our choice of $\hat\lambda$ recovers again the correct behavior of $\lambda_0$ for $v\sim 0$. In this scenario the CPDE will be close to a contact process running on a static percolation cluster of $\Z$ with parameter $p$, and in such a graph the infection is necessarily trapped inside finite components where it eventually dies out, and hence we expect $\lambda_0\longrightarrow\infty$ as $v\to0$. This is the first part of our next result:
 
 \begin{theorem}\label{teo1}
 	\leavevmode
	\begin{enumerate}[label=(\alph*)]
	  \item For all $p\in[0,1)$ we have $\lim_{v\rightarrow0}\lambda_0(v,p)=\infty$.
	  	In other words, for all $\lambda>0$ and all $p\in[0,1)$ we can take $v$ small enough so that the infection process in the CPDE dies out.
	  \item Furthermore, for $p$ fixed and $v$ small enough there are constants $\beta_0,\beta_1>0$ (which may depend on $p,\lambda$ and $v$) such that for any initial condition $\eta_0$, we have
	  	\begin{equation*}
	  	\E(\tau_\text{\normalfont ext}\,|\,\eta_0)\;\leq\;\beta_0\log(|\eta_0|)+\beta_1
	  	\end{equation*}
	  	where $|\eta_0|$ stands for the number of initially infected sites and $\tau_\text{\normalfont ext}$ is the extinction time of the CPDE (i.e. $\tau_\text{\normalfont ext}=\inf\!\big\{t\geq0\!:\eta_t=\emptyset\big\}$).
 \end{enumerate} \end{theorem}
 
\begin{remark}\label{rem:v}
 It is reasonable to guess that $\lambda_0(v,p)$ should be monotone in $v$: in fact, if one thinks of the dynamics of the CPDE as having each infected site send an infection at rate $\lambda$ to a randomly chosen neighbor and then discarding those infections going through unavailable edges or landing on already infected sites, then when $v$ is small the network does not change much, so many attempted infections are wasted, while larger values of $v$ make it easier for infected sites  to reach previously unreachable healthy sites.
 Our previous results together with Theorem \ref{teo1}(a) (as well as some of the results following below) are, at least, consistent with this view.

 \noindent However, Theorem \ref{teo1}(b) shows that the behavior of the CPDE is necessarily subtler than what this guess would suggest, at least for small $v$.
 Indeed, for any $\lambda>\bar{\lambda}$ it is known that the standard contact process running on $\{1,\dotsc,N\}$ (starting with a single infected site) has expected extinction time bounded from below by $c_0e^{c_1 N}$, where $c_0,c_1>0$ are independent of $N$. 
 On the other hand, for the CPDE starting at any $\eta_0$ with $|\eta_0|$ large enough it is easy to show that, with a very large probability, at least one of the initially infected sites is contained on a connected interval of length $c_2\ln(|\eta_0|)$, where $c_2>0$ depends on $p$.
 It follows that in the static case we have
 \[\E_{\lambda,0,p}(\tau_\text{ext}\,|\,\eta_0)\;\geq\;c_0|\eta_0|^{c_1c_2},\]
 while for $v>0$ as in the theorem we have $\E_{\lambda,v,p}(\tau_\text{ext}\,|\,\eta_0)\;\leq\;\beta_0\log(|\eta_0|)+\beta_1$ so, for large initial conditions, the infection running on a static environment is more resilient than the same infection running on a slightly dynamical one.
\end{remark}


A natural question raised by Theorem \ref{teo1} is whether there are (small, but positive) values of $v$ and/or $p$ such that $\lambda_0(v,p)=\infty$, which means that infections always die out, regardless of their infection rates. We will say in such a case that the network is {\it immune}, and we define the {\it immunity region} $\mathfrak{I}$ accordingly as
 \[\mathfrak{I}=\big\{(v,p)\in(0,\infty)\times(0,1),\lambda_0(v,p)=\infty\big\}.\]
Note that all finite static graphs are immune while all infinite connected static graphs are not, so the existence of a non-trivial immunity region would show that, in a sense, our dynamic random graph $\{(\Z,E_t)\}_{t\geq0}$ lies halfway between the two cases (all connected clusters are finite at any given time, but for any two sites $x$ and $y$ and any given $t>0$ there is a.s. some path connecting $(x,t)$ and $(y,s)$ for $s>t$ sufficiently large). Our next result shows that if instead of taking $v$ small as in Theorem \ref{teo1} we take $p$ close to zero, then we do have immunity, and in this scenario the expected extinction time also grows at most logarithmically with the initial condition:
 
 \begin{theorem}
 	\label{teo2}
 	\leavevmode
 	\begin{enumerate}[label=(\alph*)]
 		\item For all $v>0$ there is a $p_0(v)\in(0,1)$ such that $\lambda_0(v,p)=\infty$ for all $p<p_0(v)$.
 		In other words, there is a curve $v\in(0,\infty)\longmapsto p_0(v)\in(0,1)$ such that $\{(v,p):0<p<p_0(v)\}\subseteq\mathfrak{I}$.
 		\item Furthermore, for $v$ fixed and any $p<p_0(v)$ there are constants $\beta_0,\beta_1>0$ (which may depend on $p$ and $v$, but not on $\lambda$) such that for any initial condition $\eta_0$, we have
 		\begin{equation*}
 		\E(\tau_\text{\normalfont ext}\,|\,\eta_0)\;\leq\;\beta_0\log(|\eta_0|)+\beta_1
 		\end{equation*}
 		where $|\eta_0|$ and $\tau_\text{\normalfont ext}$ are as in Theorem \ref{teo1}.
 	\end{enumerate} 
 \end{theorem}
 
 Note that we are saying that, no matter how large we take $v$, a small enough density $p$ of open edges yields immunity.
 This may seem to be slightly counterintuitive, and in particular it seems to contradict our intuition for the case $v\gg1$, where we argued that $\lambda_0\sim\frac{\bar{\lambda}}{p}$. Notice, however, that in that argument we took $p$ fixed and $v\rightarrow\infty$ instead of $v$ fixed and $p\searrow0$ (see Figure \ref{fig1}). In the opposite direction, since small values of $v$ can be seen as hurting the infection (Theorem \ref{teo1}(a)), it is natural to ask whether for any $p$ one can find a small enough $v$ which guarantees immunity.
 The next theorem provides a (partial) negative answer to this question:
 
 \begin{theorem}
 	\label{teo3}
 	There exists $0<p_1<1$ such that for all $p\in(p_1,1)$, $\lambda_0(v,p)<\infty$ for all $v>0$.
 \end{theorem}
 
 From Proposition \ref{monotoneprop} we know that $\lambda_0(v,p)=\infty$ implies $\lambda_0(v',p')=\infty$ for all $v'\leq v$ and $p'\leq p$.
 As a consequence of this and Theorems \ref{teo2} and \ref{teo3}, we obtain the following:
 
 \begin{corollary}
 	There exists $p_1\in(0,1)$ such that for every $p>p_1$, $\lambda_0(v,p)<\infty$ for all $v>0$, while for every $p<p_1$, there is a $v>0$ with $\lambda_0(v,p)=\infty$.
 \end{corollary}
 
 In words, there exists a threshold parameter $p_1$ (note $p_1=\lim_{v\to0}p_0(v)$ with $p_0(v)$ as in Theorem \ref{teo2}) which separates the scenario where there is immunity for small $v$, and the one where immunity cannot occur in the network. From these considerations, one expects the upper boundary of $\mathfrak{I}$ to be the graph of a decreasing function, see Figure \ref{fig1}.
 \begin{figure}[h!]
 	\centering
 	\begin{tikzpicture}[line cap=round,line join=round,>=triangle 45,x=0.8cm,y=0.8cm]
 	
 	\node  at (0.0,2.5)[left]{$p$};
 	\fill [yellow, domain=0:4.85, variable=\x]
 	(0, 0)
 	-- plot ({\x}, {6.4/(\x+2)^2})
 	-- (4.85, 0)
 	-- cycle;
 	\draw [line width=1.2pt] (0.0,0.0)-- (0.0,2.5);
 	\draw [line width=1.2pt,->] (-0.05,0.0)node[left]{$0$}-- (5.2,0.0)node[right]{$v$};
 	\draw [line width=1pt] (0.0,2.0) -- (-0.05,2.0)node[left]{$1$};
 	\draw [line width=0.7pt, dashed] (0.0,1.6)node[left]{$p_1$} -- (4.85,1.6);
 	\node at (0.6,0.4){$\mathfrak{I}$};
 	\end{tikzpicture}
 	\caption{Expected shape of $\mathfrak{I}$.}
 	\label{fig1}
 \end{figure}
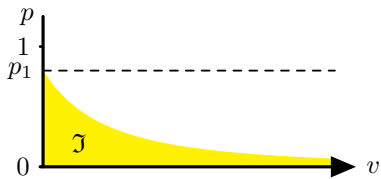

 \subsection{Extensions to vertex-transitive regular graphs}\label{sec:genG}

 Let $(\G,E)$ be any infinite vertex-transitive graph with (finite) constant degree.
 The CPDE can be defined on $\G$ exactly as in Section \ref{sec:CPDEonZ}, and the definition of $\lambda_0(v,p)$ is the same in this case.
 As we mentioned at the end of the introduction, most of our results are valid in this more general case.
 We state this as follows:

 \begin{theorem}
 All of the results of Section \ref{sec:results} except for Theorem \ref{teo1} are valid for the CPDE running on $\G$ (after replacing the parameter $\bar\lambda$ by the critical value $\clambda$ of the static contact process running on $\G$).
 \end{theorem}

We will prove the results mentioned in this theorem in Section \ref{proofs} in the more general setting of vertex-transitive regular graphs.
The only exception is Theorem \ref{teo3} which we will only prove on $\Z$, since it is enough to do so because any infinite vertex-transitive graph $\G$ contains a copy of $\Z$.

Our proof of Theorem \ref{teo1}, on the other hand, uses one-dimensional methods, so it cannot be extended directly to general $\G$.
We believe, although we have not pursued this any further, that a version of (a) in that result should hold in the general setting; more precisely, one expects that as $v\to0$ the critical parameter $\lambda_0(v,p)$ should converge to the critical $\lambda$ for the (static) contact process running on an edge percolation cluster on $\G$ with parameter $p$ (note that the critical $p$ for edge percolation on $\Z$ is $p_c=1$, so this is consistent with Theorem \ref{teo1}(a)).

 \section{Proofs}
 \label{proofs}
 
 \subsection{Graphical representation}
 \label{graphical}
 
 In this section, we provide an equivalent description of our model by a convenient graphical representation with the help of the following independent Poisson point processes on $(0,\infty)$:
 \begin{itemize}
 	\item $\{\mO^e\}_{e\in E}$ and $\{\mC^e\}_{e\in E}$, with intensities $vp$ and $v(1-p)$ respectively. These represent the opening and closing events of the edge $e$. We also consider $\mU^e=\mO^e\cup\;\mC^e$, the updating events of $e$. 
 	\item $\{\mI^e\}_{e\in E}$, with intensity $\lambda$. These represent potential infection events along the edge $e$.
 	\item $\{\mR^x\}_{x\in\Z}$, with intensity 1. These represent recovery events.
 \end{itemize}
 We construct the environment process $\zeta_t$ by choosing $\zeta_0$ according to a Bernoulli product measure $\pi$ with parameter $p$, and then setting $\zeta_t(e)=1$ if and only if the last updating event in $\mU^e\cap[0,t]$ was in $\mO^e$ (or if $\zeta_0(e)=1$ in case the intersection is empty). The infection process $\eta_t$ is then constructed in the usual manner (see \cite[Sec. 3.6]{IPS1}), using the recovery events $\mR^x$ and the \textit{valid} infection events $\mbI^e=\{t\in \mI^e,\;\zeta_t(e)=1\}$ as follows: consider the \textit{graphical space} $\Z\times\R^+\subseteq\R\times\R^+$, with the addition of horizontal segments of the form $[x,x+1]\times\{t\}$ with $t\in\mbI^{\{x,x+1\}}$. By assigning a $*$ symbol at each point $(x,t)$ with $x\in\Z$ and $t\in \mR^x$, we say that a continuous path $P$ in this space is {\it valid} if it does not contain $*$ symbols and if the second component is non-decreasing. We define $\eta_t$ as the set of all $z\in\Z$ such that there is a valid path $P$ from initially infected sites to $(z,t)$.
 
 This construction will be used persistently throughout the paper, providing simple and intuitive couplings between the CPDE and other processes which are easier to analyze.
 
 \subsection{Proof of Proposition \ref{monotoneprop}}
 
 Suppose that the infection process survives with positive probability for given $\lambda,v,p$ and take any $v'>v$; rescaling time by a factor of $v/v'$ gives a process constructed in the same way as the CPDE but where the Poisson point processes $\mI^e,\;\mO^e,\;\mC^e$ and $\mR^x$ have intensities $\frac{v'}{v}\lambda,\;v'(1-p),\;v'p$ and $v'/v$ respectively. Since $v'/v>1$, we can couple $\mR^x$ with a Poisson point process $\bar{\mR}^x$ with intensity rate 1 in such a way that $\bar{\mR}^x\subseteq\mR^x$; the process constructed with $\bar{\mR}^x$ instead of $\mR^x$ is a CPDE with parameters $\frac{v'}{v}\lambda,\;v'$ and $p$. From the coupling it is obvious that survival is easier when replacing ${\mR}^x$ by $\bar{\mR}^x$, so $\lambda_0(v',p)\leq\frac{v'}{v}\lambda$, and since this inequality holds for all $\lambda>\lambda_0(v,p)$ it holds for $\lambda_0(v,p)$ as well, giving the result.
 
 \subsection{Proof of Proposition \ref{lambda1}}
 
 The lower bound for $\lambda_0(v,p)$ comes simply from comparing with $\lambda_0(v,1)$, or in other words with a contact process constructed directly using $\{\mI^e\}_{e\in E}$ and $\{\mR^x\}_{e\in E}$ respectively as the infection and recovery events. For the contact process dominating the CPDE from below, which yields the upper bound, fix any $e\in E$ and use the events $\{\mO^e\}_{e\in E}$, $\{\mC^e\}_{e\in E}$, and $\{\mI^e\}_{e\in E}$ in the graphical representation to construct a process $\{(\zeta_t(e),N_t(e))\}_{t\geq 0}$ where $\zeta_t(e)$ is the environment defined above and $N_t(e)=|\mbI^e\cap[0,t]|$ is the number of valid infections occuring at $e$ up to time $t$. It can be easily checked that this process is Markov and its transition rates are of the form
 \begin{align*}
 (0,k)&\longrightarrow (1,k)&&\hskip-1.2in \mbox{ at rate }\quad vp,\\
 (1,k)&\longrightarrow(0,k)&&\hskip-1.2in \mbox{ at rate }\quad v(1-p),\\
 (0,k)&\longrightarrow (0,k+1)&&\hskip-1.2in \mbox{ at rate }\quad 0,\\
 (1,k)&\longrightarrow(1,k+1)&&\hskip-1.2in \mbox{ at rate }\quad \lambda.
 \end{align*}
 From \cite[Thm. 1.4]{broman2007}, $(\zeta_t(e),N_t(e))$ can be coupled with a Poisson process $P_t(e)$ with intensity $\beta(\lambda,v,p)$, in such a way that every jump of $P_t(e)$ coincides with a jump of $N_t(e)$. It follows that $P_t(e)$ defines a Poisson process $\mathcal{P}^e\subseteq \mbI^e$ and from the independence of the processes across edges, the family $\{\mathcal{P}^e\}_{e\in E}$ is independent. The result follows from constructing a contact process with the $\mathcal{P}^e$ marking infection events and the $\mathcal{R}^x$ marking recoveries.
 \begin{remark}
 	A similar argument appears in the proof of Theorem 1(c) of the published version of \cite{Remenik}, but unfortunately that proof is flawed, and in fact the argument cannot be applied in the setting of that paper (see the updated arXiv version cited in \cite{Remenik}).
 \end{remark}
 
 \subsection{Proof of Theorem \ref{teofast}}
 
 Fix $p\in(0,1]$. From Proposition \ref{lambda1} we already deduced that if $\lambda>\bar{\lambda}/p$, then for any $v$ sufficiently large the process survives, so to prove Theorem \ref{teofast} we need only to show that if $\lambda<\bar{\lambda}/p$, then for $v$ large enough the CPDE dies out. The key idea of the proof is that if we fix an edge $e\in E$ and call $t_1,t_2,\dotsc$ the elements of $\mI^e$ in increasing order, then if $v$ is large most intervals $(t_i,t_{i+1})$ will contain an updating event $s\in \mU^e$, and, conditional on that, the infection $t_{i+1}$ is valid with probability $p$ independently of all previous infection events (and from the ones taking place at different edges). At a heuristic level, this means that we can treat $\eta_t$ as the usual contact process on $\G$ with rate $\lambda p<\clambda$, which is subcritical.
 
 In order to turn this heuristic into an actual proof we need to control the infection events that do not satisfy the property stated above and show that these cannot account for survival of $\eta_t$. We keep track of these infections with the aid of a sequence of processes $(F_t(e))_{e\in E}$ defined as follows: 
 \begin{defn}
 	For each $e\in E$ define a {\cad} process $F_t(e)$ with values in $\{0,1\}$ which starts with $F_0(e)=0$ and jumps at times $t\in \mI^e\cup \mU^e$ with
 	\[F_{t}(e)\;=\;\begin{dcases*}0& if $t\in\mI^e$,\\1& if $t\in\mU^e$.\end{dcases*}\]
 	We say that at time $t>0$ the edge $e$ is \textit{fresh} if $F_{t^-}(e)=1$.
 \end{defn}
 In other words, $F_t$ serves as an indicator function of the set of edges whose latest event in $[0,t)$ is an update. In order to control the infections taking place at unrefreshed edges we will actually work with a ``worst-case scenario'' process $\eL$ in which all infections taking place at unrefreshed edges are treated as valid.
 
 \begin{defn}
 	We say that an infection event $t\in\mI^e$ is \textit{weakly valid} if either $F_{t^-}(e)=0$ or $F_{t^-}(e)=1$ and $t\in\mbI^e$. \textit{Weakly valid paths} are defined analogously to valid paths in Section \ref{graphical}, but instead of using only valid infections for the paths to move horizontally, we use weakly valid ones. The process $(\eL_t)_{t\geq0}$ is defined analogously to the CPDE starting with an initially infected site at $\{0\}$, where $\eL_t(x)=1$ if and only if there is a weakly valid path from $(0,0)$ to $(x,t)$.
 \end{defn}
 
 Any valid infection is also weakly valid, so $\eta_t\leq\eL_t$, and hence in order to show that $\eta_t$ dies out it suffices to show that $\eL_t$ does. We will do this by studying a third process $\ec$ which we define using an extension of the graphical construction, and which evolves as the desired contact process with rate $\lambda p$:
 
 \begin{defn}
 	Consider an enlarged version of the graphical construction in which we split each $\mI^e$ into two independent Poisson processes, $\mIA^e$ and $\mIR^e$ with rates $\lambda p$ and $\lambda(1-p)$ respectively. We say that an infection event $t\in\mI^e$ is \textit{$p$-weakly valid} if either $F_{t^-}(e)=1$ and $t\in\mbI^e$, or if $F_{t^-}(e)=0$ and $t\in\mIA^e$. In words, an infection taking place at a fresh edge is $p$-weakly valid if valid, and at an unrefreshed edge we flip a coin to decide.
 	\textit{$p$-weakly valid} paths are defined analogously to weakly valid paths, and with them we construct a process $(\ec_t)_{t\geq0}$ analogously to $\eL$, starting with an initially infected site at $\{0\}$.
 \end{defn}
 
 It follows from its definition that any infection event is $p$-weakly valid with probability $p$ independently of all previous infections so that, as desired, $\ec$ evolves as the (usual) contact process with rate $\lambda p$. On the other hand, notice that every $p$-weakly valid infection is also weakly valid, so the processes above satisfy $\ec_t\leq\eL_t$ for each $t\geq0$; in fact, the two processes essentially drift apart only at times $\tau_1,\tau_2,\dotsc$ at which $\eL$ propagates to healthy sites but $\ec$ does not. We formally introduce said times as follows:
 
 \begin{defn}\label{tiempos}
 	Take $\tau_0=0$ and for each $k\geq 1$ define
 	\[\tau_{k}\,=\;\inf\{t>\tau_{k-1},\;\exists e=\{x,y\}\in E,\;\ec_t(x)=0,\,\ec_t(y)=1\text{ with }F_{t^-}(e)=0\text{ and }t\in\mIR^e\}.\]
 	We use the notation $x_k$ to denote the vertex in $e$ at time $\tau_k$ for which $\ec_{\tau_k}(x)=0$. We also use the notation $N_p$ to refer to the largest $k\in\N\cup\{\infty\}$ such that $\tau_k<\infty$.
 \end{defn}
 
 Observe that any $\tau_k$ is associated to an infection taking place at an unrefreshed edge, so as $v\to\infty$ we expect to have $N_p=0$, which at a heuristic level would give $\eL\approx\ec$. This in turn allows us to control $\eL$, since we already know that $\ec$ behaves as a subcritical contact process from our choice of $\lambda$. The next lemma, which formally states our claims about $N_p$ and $\ec$, will be key in the proof of the theorem: 
 
 \begin{lemma}\label{keyrestricted}
 	Fix $\lambda<\bar{\lambda}/p$. For any $\zeta_0\in\{0,1\}^E$ we have
 	\begin{equation*}\label{death_er}
 	\P_{\zeta_0}(\ec_t\neq\emptyset\,\,\,\forall t>0)\;=\;0,
 	\end{equation*}
 	where $\P_{\zeta_0}$ stands for the law of the process with initial environment configuration $\zeta_0$. Furthermore,
 	\begin{equation*}
 	\lim_{v\to\infty}\sup_{\zeta_0}\E_{\zeta_0}(N_p)=0,
 	\end{equation*}
 	where the supremum is taken over all initial configurations $\zeta_0\in\{0,1\}^E$.
 \end{lemma}
 
 Using this final ingredient, whose proof we defer to the end of the section, we are now able to prove Theorem \ref{teofast}. We will actually show something a little bit stronger, namely that for every initial condition $\zeta_0$ we have that $\eL$ dies out, i.e. that
 \begin{equation}
 \P_{\zeta_0}\big(\forall t\geq0,\,\eL_t\neq\emptyset\big)\;=\;0.\label{Ptobd}
 \end{equation}
 Define the event
 \[A=\{\exists t_0\geq0,\,\ec_t=\emptyset\;\forall\,t\geq t_0\},\]
 which by Lemma \ref{keyrestricted} occurs with probability 1; in particular, the left hand side of \eqref{Ptobd} equals \smash{$\P_{\zeta_0}\big(\{(0,0)\lL\infty\}\cap A\big)$}, where for $y\in\Z$ and $t\geq0$ the event $(y,t)\lL\infty$ stands for the existence of a weakly valid path $P$ starting at $(y,t)$ which is unbounded in its time component.
 Take a realization of the extended graphical construction and suppose $P$ is such a path (so that $(0,0)\lL\infty)$. Then, on $A$, $P$ must traverse a weakly valid infection event $(\{x,y\},t)$ which is not $p$-weakly valid, and such that $\ec_t(x)=0$ and $\ec_t(y)=1$ (since, otherwise, $P$ would also count as a weakly valid path). 
 Hence the left hand side of \eqref{Ptobd} is equal to
 \[\P_{\zeta_0}\big(\{\exists k\in\N,\;\tau_k<\infty\text{ and }(x_k,\tau_k)\lL\infty\}\cap A\big)\leq\sum_{k=1}^\infty \P_{\zeta_0}\big(\tau_k<\infty\text{ and }(x_k,\tau_k)\lL\infty\big).\]
 The $\tau_k$ are stopping times so by the strong Markov property we get
 \[\P_{\zeta_0}\big(\tau_k<\infty\text{ and }(x_k,\tau_k)\lL\infty\big)\;=\;\E_{\zeta_0}\big(\one_{\{\tau_k<\infty\}}\P\big((x_k,0)\lL\infty\,|\,\zeta_{\tau_k},F_{\tau_k}\big)\big).\]
 But $\eL$ is decreasing with respect to $F_0$, so by taking $F_0\equiv 0$ and then taking the supremum with respect to $\zeta_0$, we can use the translation invariance (in law) of $\eL$ to deduce
 \begin{align*}
 \P_{\zeta_0}\big(\tau_k<\infty\text{ and }(x_k,\tau_k)\lL\infty\big)&\leq\P_{\zeta_0}\big(\tau_k<\infty)\sup_{\zeta'_0}\P_{\zeta'_0}\big((0,0)\lL\infty\big)\\
 &=\P_{\zeta_0}(\tau_k<\infty)\sup_{\zeta'_0}\P_{\zeta'_0}\big(\eL_t\neq\emptyset,\;\forall t\geq0\big).
 \end{align*}
 Now, $\tau_k<\infty$ if and only if $N_p\geq k$, so from the above arguments we get
 \begin{align*}
 \P_{\zeta_0}\big(\eL_t\neq\emptyset,\;\forall t\geq0\big)&\leq\sum_{k=1}^\infty\P_{\zeta_0}(N_p\geq k)\sup_{\zeta'_0}\P_{\zeta'_0}\big(\eL_t\neq\emptyset,\;\forall t\geq0\big)\\
 &=\E_{\zeta_0}(N_p)\sup_{\zeta'_0}\P_{\zeta'_0}\big(\eL_t\neq\emptyset,\;\forall t\geq0\big),
 \end{align*}
 and taking the supremum over $\zeta_0$ we conclude that
 \[\sup_{\zeta_0}\P_{\zeta_0}(\eL_t\neq\emptyset,\;\forall t\geq0)\;\leq\;\sup_{\zeta_0}\E_{\zeta_0}(N_p)\sup_{\zeta'_0}\P_{\zeta'_0}\left(\eL_t\neq\emptyset,\;\forall t\geq0\right).\]
 But from Lemma \ref{keyrestricted} we know that if $v$ is large then $\sup_{\zeta_0}\E_{\zeta_0}(N_p)<1$, so the last inequality gives $\sup_{\zeta_0}\P_{\zeta_0}(\eL_t\neq\emptyset,\;\forall t\geq0)=0$, proving the theorem.
 
 \medskip

\begin{proof}[Proof of Lemma \ref{keyrestricted}]
Observe that if the initial configuration for the environment were to be chosen at random, then from a previous discussion the law of $\ec$ would be that of a contact process with rate $\lambda p$ and hence it would die out. However, since we start with a fixed, given $\zeta_0$ and $F_0\equiv 0$, the first infection event at every edge could have a higher chance to be a $p$-weakly valid one, which means that in $\ec$ the time until the first infection event in each edge has a different distribution; our goal then is to show that this feature cannot account for survival. To this end fix any $v>1$ and choose $\varepsilon>0$ small so that $\lambda (p+\varepsilon)<\bar{\lambda} $ (recall we are assuming $\lambda p<\bar{\lambda}$). Next, take $s=\log({1-p\over\varepsilon})$ (which is of course positive if $\varepsilon<1-p$) and finally fix any initial condition $F_0$ and $\zeta_0$. To show that $\ec$ dies out we bound it from above by a process $\bar{\eta}^p$ defined as follows:
  \begin{itemize}
  	\item From times $0$ to $s$, $\bar{\eta}^p$ evolves as a SI process with rate $\lambda$, i.e. without recoveries and behaving as if all edges are open, starting with only one infected site at $0$. During this time interval the process is constructed using only $\mI$.
  	\item From time $s$ onwards, $\bar{\eta}^p$ is constructed in the same way as $\ec$.
  \end{itemize}
 Observe that $|\bar{\eta}^p_s|$ has finite expectation (since $\G$ has bounded degree) and it is independent of $\zeta$. Also, for any $s'\geq s$ the law of $\zeta_{s'}$ is a product measure with
 \[\P(\zeta_{s'}(e)=1)\;=\;\zeta_0(e)\exp(-vs')+p(1-\exp(-vs'))\;\leq\;p+\varepsilon,\]
 from our choice of $s$ and our assumption $v>1$. Now, it follows that each infection event after time $s$ has probability at most $p+\varepsilon$ of being a $p$-weakly valid one, independently of all other infection events, so that $(\bar{\eta}^p_{s+t})_{t\geq 0}$ is bounded from above by a contact process running on $\G$ with rate $\lambda(p+\varepsilon)$ and with initial condition $\bar{\eta}^p_s$. Since $\G$ is vertex-transitive it follows from Theorem 1.2. in \cite{aizen} that 
 for any subcritical contact process $A_t$ we have
 \[\int_{0}^\infty \E\big(|A_t|\,\big|\,A_0=\{0\}\big)dt\,<\,\infty,\]
 so from a union bound, translation invariance and the independence between $\bar{\eta}^p_s$ and $\zeta$, we conclude similarly that
 \begin{align*}
 	\int_0^\infty\E\big(|\bar{\eta}^p_t|\big)dt\;&\leq\;s\E(|\bar{\eta}^p_s|)+\E\bigg(\int_0^\infty\E\big(|\bar{\eta}^p_{s+t}|\,\big|\,\bar{\eta}^p_s\big)\,dt\bigg)\\
 	&\;\leq s\E(|\bar{\eta}^p_s|)+\E\bigg(\int_0^\infty|\bar{\eta}^p_s|\,\E\big(|\bar{\eta}^p_{s+t}|\,\big|\,\bar{\eta}^p_s=\{0\}\big)\,dt\bigg)\\
 	&\;=\;\E(|\bar{\eta}^p_s|)\bigg(s+\int_0^\infty\E\big(|\bar{\eta}^p_{s+t}|\,\big|\,\bar{\eta}^p_s=\{0\}\big)\,dt\bigg)\,<\,\infty,
 \end{align*}
 which in particular gives that $\eta^p$ dies out almost surely.

 We turn now to the second part of the lemma. We begin by noticing that the expectation is maximized when $\zeta_0\equiv 1$ so it is enough to prove the result under this initial condition, which we now fix. Define an increasing sequence of events
 \[A_n\;=\;\big\{\forall (x,t)\notin B(0,n)\times[0,n],\;\ec_t(x)=0\big\},\]
 and then use the fact that $\ec$ dies out almost surely for all initial configurations to write
 \begin{equation*}
 \E(N_p)\;=\;\sum_{n=1}^\infty\E(N_p\uno{A_{n}\setminus A_{n-1}}).\label{eq:eqsum}
 \end{equation*}
 As before, observe that each $\tau_k$ corresponds to an infection event taking place at an unrefreshed edge, and that, on the event $A_n$, $\ec_t(x)=1$ implies that $(x,t)\in B(0,n)\times[0,n]$ so that $N_p$ is bounded by the amount $M_n$ of said infections taking place inside $B(0,n)\times[0,n]$. Thus
 \begin{equation*}\E(N_p\uno{A_{n}\setminus A_{n-1}})\;\leq\;\E(M_n).\label{desmn}\end{equation*}
 We claim that each such term goes to zero as $v\to\infty$. Indeed, fix some $n\in\N$ and take $v>16\lambda^2n^2$. Next, divide $M_n$ into $M_n=\cup_{e\in E\cap B(0,n)} M_{n,e}$, where each $M_{n,e}$ corresponds to the number of infections occuring at edge $e$ at times when $e$ was not fresh and before time $n$. By translation invariance we obtain that all the $M_{n,e}$ have the same law, so $\E(M_n)\leq|B(0,n)|\E(M_{n,e})$ for any fixed edge $e$. Call $t_0=0$ and $t_1,t_2,\dotsc$ the elements of $\mI^e$. The variable $M_{n,e}$ is equal to the cardinality of the set $\big\{k\in\N,\;t_k\leq n\;\wedge\;\mU^e\cap(t_{k-1},t_k)=\emptyset\big\}$, so its expectation is equal to
 \[\sum_{k=1}^\infty\P\big(t_k\leq n\;\wedge\;\mU^e\cap(t_{k-1},t_k)=\emptyset\big).\]
 If $k\leq\sqrt{v}$ we bound the above probability by $\P\big(\mU^e\cap(t_{k-1},t_k)=\emptyset\big)$, which is the probability that the next event in $\mI^e\cup\mU^e$ following $t_{k-1}$ belongs to $\mI^e$, and hence is equal to $\lambda/(v+\lambda)$.
 Otherwise, if $k>\sqrt{v}$ we bound by $\P\big(t_k\leq n\big)$, which by a large deviation argument and our assumption on $v$, is less than $(e/4)^k$. Using these bounds, we conclude that, as desired,
 \begin{equation*}
 \E(M_n)\;\leq\;|B(0,n)|\bigg[\frac{\lambda\sqrt{v}}{\lambda+v}+4\left(\frac{e}{4}\right)^{\sqrt{v}}\bigg]\xrightarrow[v\to\infty]{}0.\label{eq:eMnto0}
 \end{equation*}
 Since each $\E(N_p\uno{A_{n}\setminus A_{n-1}})$ converges to zero, in order to finish the proof it is enough (by the Dominated Convergence Theorem) to show that these expectations can be bounded by the terms of a convergent series. To achieve this we construct a random variable $\bar{N}_p$, similar to $N_p$, as follows: $\bar{N}_p$ equals the cardinality of the set of all infection events $t\in\R^+$ that are not $p$-weakly valid, and taking place at edges $\{x,y\}\in E$ for which either $\eta_t(x)=1$ or $\eta_t(y)=1$. We have $N_p\leq \bar{N}_p$ almost surely, so
 \[\E(N_p\uno{A_{n}\setminus A_{n-1}})\;\leq\;\E(\bar{N}_p\uno{A_{n}\setminus A_{n-1}})\]
 for each $n$ and hence all we need to show is that $\sum_{n\in\N}\E(\bar{N}_p\uno{A_{n}\setminus A_{n-1}})=\E(\bar{N}_p)<\infty$. Observe that for each edge $e\in E$ the first $t\in\mI^e\cup\mU^e$ is a $p$-weakly valid infection event if and only if $t\in\mI^e$, while all subsequent times are $p$-weakly valid with probability $p$, independently from one another. Hence the $p$-weakly and non $p$-weakly valid infections can be seen as resulting from the following construction:
 \begin{itemize}
 	\item For each $e\in E$ consider an exponential time $t_{e}$ with rate $\lambda+v$, a Bernoulli random variable $c_e$ with probability $\lambda/(v+\lambda)$, and two Poisson point processes $\mI^e_\text{wv}$ and $\mI^e_\text{wi}$ with rates $\lambda p$ and $\lambda(1-p)$, respectively. All of these variables and processes are independent from one another.
 	\item The $p$-weakly valid infections are the elements $t\in\mI^e_\text{wv}$ with $t>t_e$; $t_e$ itself is $p$-weakly valid if $c_e=1$.
 	\item The non $p$-weakly valid infections are the elements $t\in\mI^e_\text{wi}$ with $t>t_e$.
 \end{itemize}
 With this alternative construction, it follows that $\ec$ is constructed using the $\mathcal{R}^x$, $\mI^e_\text{wv}$, $t_e$ and $c_e$ but not the $\mI^e_\text{wi}$. Fix a realization of the former processes giving a construction of $\ec$ and observe that
 \begin{align*}
 \E(\bar{N}_p\,|\,\ec)\,&\leq\,\sum_{x\in \G}\sum_{y\sim x}\E\big(|\{t\in\mI^{\{x,y\}}_\text{wi},\;\eta_t(x)=1\}|\,\big|\,\ec\big)\leq\,\sum_{x\in \G}\sum_{y\sim x}\lambda(1-p)\int_0^\infty\ec_t(x)\,dt\\
 &=\,\lambda(1-p)\,\text{deg}(\G)\int_0^\infty|\ec_t|\,dt,
 \end{align*}
 where in the second inequality we used that the $\mI^e_\text{wi}$ are Poisson processes, so that each variable $|\{t\in\mI^{\{x,y\}}_\text{wi},\;\eta_t(x)=1\}|$ is Poisson with mean $\lambda(1-p)\int_0^\infty\ec_t(x)\,dt$.
 We deduce that $\E(\bar{N}_p)\leq\lambda(1-p)\,\text{deg}(\G)\int_0^\infty\E(|\ec_t|)\,dt$, which is finite from our previous analysis on $\bar{\eta}^p$.
 \end{proof}

\subsection{Proof of Theorem \ref{teo1}}

In this part we work on $\Z$.
In this case, and as we mentioned above, our evolving networks can be thought of as lying halfway between a finite and an infinite graph: even though at all times $\Z$ is partitioned into finite components, every two sites are eventually connected by space-time paths.
The proof of Theorem \ref{teo1} relies on showing that in this regime the finite aspect of the evolving network dominates.
The idea is simple: for small enough values of $v$, in the time scale of the infection, typical connected components look almost static, so the process becomes extinct within them; exceptionally large components, on the other hand, are unstable, quickly dividing into smaller ones, so they cannot account for survival.

To define what ``typical connected components'' are in a useful way, consider the set $r_0\Z$ where $r_0\in\N$ is some large integer to define later as a function of $p$. The main idea is to partition $\Z$ into intervals around a family of elements chosen from this set in such a way that we can control the infection inside them.
However, since our network evolves in time, we will need to allow these blocks to evolve as well.
To this end we partition $\R^+$ into time intervals of the form $[nT,(n+1)T)$, with $T>0$ a large parameter to be fixed later as a function of $\lambda$ and $r_0$; our space-time blocks will always have the form $B_{k,n}\times[nT,(n+1)T)$ with $\big(B_{k,n}\big)_{k\in\Z}$ spatial blocks partitioning $\Z$, which depend on the time parameter $n$ and are constructed as follows.

Say that an edge $e\in E$ is {\it $n$-closed} if $\zeta_t(e)=0$ for all $t\in[nT,(n+1)T)$; this means that $e$ acts as a barrier for the infection throughout the whole time interval. Using these barriers we introduce random variables $V_{\{k,k+1\},n}$ as
\[V_{\{k,k+1\},n}={\bf 1}_{\text{no edge $e$ in $[kr_0,(k+1)r_0]$ is $n$-closed}},\]
so that $V_{\{k,k+1\},n}=0$ indicates that the infection cannot spread between $kr_0$ and $(k+1)r_0$ during the time interval $[nT,(n+1)T)$. Next let $e_{k,n}$ be the leftmost $n$-closed edge between $kr_0$ and $(k+1)r_0$, if there is some, and $e_{k,n}=\{(k+1)r_0-1,(k+1)r_0\}$ if there is none, that is, $e_{k,n}$ is the last edge traversed when moving to the right of $k r_0$ until either hitting a barrier or $(k+1)r_0$. Using these variables we finally let $B_{k,n}=[e^+_{k-1,n},e^-_{k,n}]$, where $e^-$ and $e^+$ represent the left and right vertices of $e$, respectively; see Figure \ref{fig2} for a picture.

\begin{figure}[htbp]
	\begin{center}
		\begin{tikzpicture}[line cap=round,line join=round,>=triangle 45,x=1.4cm,y=1.4cm]
		\fill [lightgray] (-2.6,0.0) rectangle (-2.4,0.15);
		\fill [lightgray] (-2.6,0.715) rectangle (-2.4,0.725);
		\fill [lightgray] (-2.6,1.055) rectangle (-2.4,1.56);
		\fill [lightgray] (-2.6,1.72) rectangle (-2.4,1.86);
		\fill [lightgray] (-2.6,2.02) rectangle (-2.4,2.35);
		\fill [lightgray] (-2.4,0.0) rectangle (-2.2,0.4);
		\fill [lightgray] (-2.4,0.6) rectangle (-2.2,0.75);
		\fill [lightgray] (-2.4,1.64) rectangle (-2.2,1.75);
		\fill [lightgray] (-2.4,1.84) rectangle (-2.2,1.89);
		\fill [lightgray] (-2.4,2.32) rectangle (-2.2,2.5);
		\fill [lightgray] (-2.2,0.3) rectangle (-2.0,0.55);
		\fill [lightgray] (-2.2,0.95) rectangle (-2.0,1.0);
		\fill [lightgray] (-2.2,1.1) rectangle (-2.0,2.5);
		\fill [lightgray] (-2.0,0.2) rectangle (-1.8,0.225);
		\fill [lightgray] (-2.0,0.2365) rectangle (-1.8,0.57);
		\fill [lightgray] (-2.0,0.835) rectangle (-1.8,0.95);
		\fill [lightgray] (-2.0,1.535) rectangle (-1.8,2.15);
		\fill [lightgray] (-1.8,0.0) rectangle (-1.6,0.875);
		\fill [lightgray] (-1.8,0.965) rectangle (-1.6,1.47);
		\fill [lightgray] (-1.8,2.254) rectangle (-1.6,2.37);
		\fill [lightgray] (-1.8,2.44) rectangle (-1.6,2.5);
		\fill [lightgray] (-1.6,0.415) rectangle (-1.4,0.98);
		\fill [lightgray] (-1.6,1.3) rectangle (-1.4,1.42);
		\fill [lightgray] (-1.6,1.47) rectangle (-1.4,1.56);
		\fill [lightgray] (-1.6,1.63) rectangle (-1.4,2.5);
		\fill [lightgray] (-1.4,0.115) rectangle (-1.2,0.23);
		\fill [lightgray] (-1.4,0.44) rectangle (-1.2,0.58);
		\fill [lightgray] (-1.4,0.765) rectangle (-1.2,0.93);
		\fill [lightgray] (-1.4,1.065) rectangle (-1.2,1.15);
		\fill [lightgray] (-1.2,0.0) rectangle (-1.0,0.1);
		\fill [lightgray] (-1.2,0.65) rectangle (-1.0,0.21);
		\fill [lightgray] (-1.2,1.58) rectangle (-1.0,1.61);
		\fill [lightgray] (-1.2,1.63) rectangle (-1.0,1.67);
		\fill [lightgray] (-1.2,1.81) rectangle (-1.0,2.01);
		\fill [lightgray] (-1.0,0.075) rectangle (-0.8,0.45);
		\fill [lightgray] (-1.0,0.925) rectangle (-0.8,1.055);
		\fill [lightgray] (-1.0,1.2) rectangle (-0.8,1.76);
		\fill [lightgray] (-1.0,2.23) rectangle (-0.8,2.5);
		\fill [lightgray] (-0.8,0.635) rectangle (-0.6,0.735);
		\fill [lightgray] (-0.6,0.0) rectangle (-0.4,0.365);
		\fill [lightgray] (-0.6,0.685) rectangle (-0.4,0.815);
		\fill [lightgray] (-0.6,0.935) rectangle (-0.4,2.5);
		\fill [lightgray] (-0.4,0.0) rectangle (-0.2,0.265);
		\fill [lightgray] (-0.4,0.45) rectangle (-0.2,1.05);
		\fill [lightgray] (-0.4,1.43) rectangle (-0.2,2.35);
		\fill [lightgray] (-0.2,0.325) rectangle (-0.0,0.835);
		\fill [lightgray] (-0.2,0.141) rectangle (-0.0,0.145);
		\fill [lightgray] (-0.2,1.46) rectangle (-0.0,1.48);
		\fill [lightgray] (-0.2,1.52) rectangle (-0.0,1.63);
		\fill [lightgray] (-0.2,1.83) rectangle (-0.0,2.07);
		\fill [lightgray] (0.0,0.1) rectangle (0.2,0.2);
		\fill [lightgray] (0.0,0.7) rectangle (0.2,0.79);
		\fill [lightgray] (0.0,0.95) rectangle (0.2,1.075);
		\fill [lightgray] (0.0,1.31) rectangle (0.2,1.71);
		\fill [lightgray] (0.2,0.515) rectangle (0.4,0.89);
		\fill [lightgray] (0.2,1.015) rectangle (0.4,2.5);
		\fill [lightgray] (0.4,0.0) rectangle (0.6,0.615);
		\fill [lightgray] (0.4,1.31) rectangle (0.6,1.415);
		\fill [lightgray] (0.4,1.52) rectangle (0.6,1.61);
		\fill [lightgray] (0.4,1.72) rectangle (0.6,1.735);
		\fill [lightgray] (0.6,0.0) rectangle (0.8,0.115);
		\fill [lightgray] (0.6,0.33) rectangle (0.8,0.545);
		\fill [lightgray] (0.6,0.8) rectangle (0.8,0.93);
		\fill [lightgray] (0.6,1.79) rectangle (0.8,2.31);
		\fill [lightgray] (0.8,1.19) rectangle (1.0,1.49);
		\fill [lightgray] (0.8,2.42) rectangle (1.0,2.5);
		\fill [lightgray] (1.0,0.27) rectangle (1.2,0.715);
		\fill [lightgray] (1.0,1.07) rectangle (1.2,1.38);
		\fill [lightgray] (1.2,0.0) rectangle (1.4,1.29);
		\fill [lightgray] (1.2,1.42) rectangle (1.4,1.45);
		\fill [lightgray] (1.2,1.48) rectangle (1.4,1.9);
		\fill [lightgray] (1.4,0.26) rectangle (1.6,0.4);
		\fill [lightgray] (1.4,0.56) rectangle (1.6,0.71);
		\fill [lightgray] (1.4,0.915) rectangle (1.6,0.985);
		\fill [lightgray] (1.4,1.315) rectangle (1.6,1.485);
		\fill [lightgray] (1.4,1.615) rectangle (1.6,2.5);
		\fill [lightgray] (1.6,0.0) rectangle (1.8,0.175);
		\fill [lightgray] (1.6,0.93) rectangle (1.8,1.225);
		\fill [lightgray] (1.6,1.73) rectangle (1.8,2.325);
		\fill [lightgray] (1.8,0.83) rectangle (2.0,0.925);
		\fill [lightgray] (1.8,1.045) rectangle (2.0,1.11);
		\fill [lightgray] (1.8,1.45) rectangle (2.0,1.56);
		\fill [lightgray] (2.0,0.0) rectangle (2.2,2.03);
		\fill [lightgray] (2.2,0.115) rectangle (2.4,0.23);
		\fill [lightgray] (2.2,0.4) rectangle (2.4,0.53);
		\fill [lightgray] (2.2,0.79) rectangle (2.4,0.945);
		\fill [lightgray] (2.2,1.05) rectangle (2.4,1.125);
		\fill [lightgray] (2.2,2.11) rectangle (2.4,2.25);
		\fill [lightgray] (2.4,0.0) rectangle (2.6,0.31);
		\fill [lightgray] (2.4,1.015) rectangle (2.6,1.165);
		\fill [lightgray] (2.4,1.215) rectangle (2.6,1.865);
		\fill [lightgray] (2.4,2.23) rectangle (2.6,2.35);
		\draw [line width=1.5pt,<->] (-3.0,0.0)-- (3.0,0.0);
		\draw [line width=1.5pt,->] (0.0,0.0)node[xshift=-0.0cm,yshift=-0.3cm]{$0$}-- (0.0,3.0);
		\draw [line width=0.3pt] (-2.6,-0.1)node[xshift=-0.1cm,yshift=-0.2cm]{$-r_0$}-- (-2.6,2.5);
		\foreach \a in {-2.4, -2.2, ..., 2.4 }
		\draw [line width=0.3pt] (\a,0.0)-- (\a,2.5);
		\draw [line width=0.3pt] (2.6,-0.1)node[xshift=-0.0cm,yshift=-0.2cm]{$r_0$}-- (2.6,2.5);
		\draw [line width=1.5pt,dotted] (-2.6,2.5)-- (2.6,2.5);
		\node at (-2.9,2.5){$2T$};
		\draw [line width=1.5pt,dotted] (-2.6,1.25)-- (2.6,1.25);
		\node at (-2.8,1.25){$T$};
		\draw[pattern=custom north west lines,hatchspread=3.5pt,hatchthickness=0.3pt,hatchcolor=red] (-2.2,1.25) rectangle (-2.0,2.5);
		\draw[pattern=custom north west lines,hatchspread=3.5pt,hatchthickness=0.3pt,hatchcolor=red] (-0.6,1.25) rectangle (-0.4,2.5);
		\draw[pattern=custom north west lines,hatchspread=3.5pt,hatchthickness=0.3pt,hatchcolor=red] (0.2,1.25) rectangle (0.4,2.5);
		\draw[pattern=custom north west lines,hatchspread=3.5pt,hatchthickness=0.3pt,hatchcolor=red] (1.2,0) rectangle (1.4,1.25);
		\draw[pattern=custom north west lines,hatchspread=3.5pt,hatchthickness=0.3pt,hatchcolor=red] (2.0,0) rectangle (2.2,1.25);
		\draw [red,line width=2.4pt] (0.0,0.0) rectangle (1.2,1.235);
		\draw [blue!85!black,line width=2.4pt] (-2.0,1.275) rectangle (0.2,2.5);
		\end{tikzpicture}
	\end{center}
	\vspace{-0.3cm}
	\caption{\footnotesize Gray rectangles represent intervals where the edge is absent, and hatched rectangles represent the barriers given by $n$-closed edges. In this case there are no $0$-closed edges between $-r_0$ and $0$, so the block $B_{0,0}\times[0,T)$ (in red) goes from $0$ to the first barrier it encounters to the right. The block $B_{0,1}\times[T,2T)$ (in blue), on the other hand, goes from the leftmost barrier between $-r_0$ and $0$, to the leftmost one between $0$ and $r_0$.}
	\label{fig2}
\end{figure}
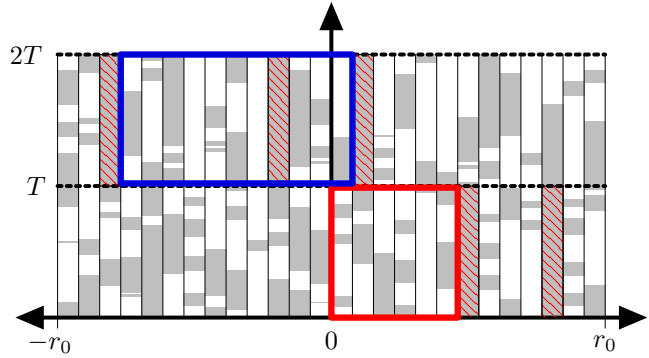

The construction of the $B_{k,n}$, which is random (depending on $\zeta$), satisfies the following properties, which are easy to check:
\begin{itemize}
	\item The blocks $\big(B_{k,n}\times[nT,(n+1)T)\big)_{k,n\in\Z}$ partition $\Z\times\R^+$.
	\item The vertex $kr_0$ always belongs to $B_{k,n}$. In particular, each $B_{k,n}$ has length between $1$ and $2r_0-1$.
	\item If $V_{\{k-1,k\},n}=V_{\{k,k+1\},n}=0$, then there are barriers separating $B_{k,n}$ from $B_{k-1,n}$ and $B_{k+1,n}$. This implies that any infection in $B_{k,n}$ gets locally {\it quarantined} during the time interval $[nT,(n+1)T)$.
\end{itemize}
The key property is the last one, since it implies that the infection should die out in $B_{k,n}$ with high probability if $T$ is large enough (compared to $r_0$). Define now
\begin{equation*}
U_{k,n}={\bf 1}_{\text{there is a valid path contained in $B_{k,n}\times[nT,(n+1)T)$ which starts at time $nT$ and ends at time $(n+1)T$}}.\label{defiU}
\end{equation*}
While the $V$ variables account for the connectivity between intervals, the $U$ variables account for the behaviour of the infection inside them; together they will give us enough information to control $\eta$. To this end we identify the blocks $B_{k,n}\times[nT,(n+1)T)$ with their respective indices $(k,n)\in\Z\times\N$ and use them as vertices of a graph $H$ whose edge set is obtained from the $U$ and $V$ variables according to the following rules:
\begin{enumerate}
	\item If $U_{k,n}=1$, then add the edges between $(k,n)$ and each of $(k-1,n+1),(k,n+1),$ and $(k+1,n+1)$.
	\item If $V_{\{k,k+1\},n}=1$, then add edges as in the previous point as if $U_{k,n}=1$ and $U_{k+1,n}=1$, and also the edge between $(k,n)$ and $(k+1,n)$.
\end{enumerate}

We say that a path between $(k_0,n_0),(k_1,n_1),\dotsc,(k_f,n_f)$ in $H$ is $H$-valid if the sequence $n_0,\dotsc,n_f$ is non-decreasing and there is no $i<f$ such that $n_i=n_f$.\\

\begin{figure}[h]
	\begin{center}
		\begin{tikzpicture}[line cap=round,line join=round,>=triangle 45,x=2.0cm,y=2.0cm]
		\foreach \i in {-2.5,-1.5,-0.5,0.5,1.5,2.5}
		{
		\draw[line width=0.9pt] (\i,0) circle (0.3cm);
		\draw[line width=0.9pt] (\i,0.7) circle (0.3cm);
		}
		\node[scale=1.3] at (-3.4,0) {$n=0$};
		\node[scale=1.3] at (-3.4,0.7) {$n=1$};
		\node[scale=1.3, red] at (-2.5,0) {$0$};
		\node[scale=1.3, red] at (-1.5,0) {$1$};
		\node[scale=1.3, red] at (-0.5,0) {$1$};
		\node[scale=1.3, red] at (0.5,0) {$0$};
		\node[scale=1.3, red] at (1.5,0) {$0$};
		\node[scale=1.3, red] at (2.5,0) {$0$};
		\node[scale=1.3, blue] at (-2,0) {$0$};
		\node[scale=1.3, blue] at (-1,0) {$0$};
		\node[scale=1.3, blue] at (0,0) {$1$};
		\node[scale=1.3, blue] at (1,0) {$0$};
		\node[scale=1.3, blue] at (2,0) {$0$};
		\draw [line width=1.1pt,dotted] (-2.8,0.0)-- (-2.66,0.0);
		\draw [line width=1.1pt,dotted] (-2.34,0.0)-- (-2.09,0.0);
		\draw [line width=1.1pt,dotted] (-1.911,0.0)-- (-1.66,0.0);
		\draw [line width=1.1pt,dotted] (-1.34,0.0)-- (-1.09,0.0);
		\draw [line width=1.1pt,dotted] (-0.911,0.0)-- (-0.66,0.0);
		\draw [line width=1.5pt] (-0.34,0.0)-- (-0.08,0.0);
		\draw [line width=1.5pt] (0.08,0.0)-- (0.34,0.0);
		\draw [line width=1.1pt,dotted] (0.66,0.0)-- (0.92,0.0);
		\draw [line width=1.1pt,dotted] (1.08,0.0)-- (1.34,0.0);
		\draw [line width=1.1pt,dotted] (1.66,0.0)-- (1.92,0.0);
		\draw [line width=1.1pt,dotted] (2.089,0.0)-- (2.34,0.0);
		\draw [line width=1.1pt,dotted] (2.66,0.0)-- (2.8,0.0);
		\draw [line width=1.5pt] (-1.635,0.0945) -- (-2.355,0.5985);
		\draw [line width=1.5pt] (-1.365,0.0945) -- (-0.645,0.5985);
		\draw [line width=1.5pt] (-1.5,0.155) -- (-1.5,0.545);
		\draw [line width=1.5pt] (-0.635,0.0945) -- (-1.355,0.5985);
		\draw [line width=1.5pt] (-0.365,0.0945) -- (0.355,0.5985);
		\draw [line width=1.5pt] (-0.5,0.155) -- (-0.5,0.545);
		\draw [line width=1.5pt] (0.365,0.0945) -- (-0.355,0.5985);
		\draw [line width=1.5pt] (0.635,0.0945) -- (1.355,0.5985);
		\draw [line width=1.5pt] (0.5,0.155) -- (0.5,0.545);
		\draw [line width=1.5pt] (0.365,0.0945) -- (-0.355,0.5985);
		\draw [line width=1.5pt] (0.635,0.0945) -- (1.355,0.5985);
		\draw [line width=1.5pt] (0.5,0.155) -- (0.5,0.545);
		\end{tikzpicture}
		\label{consth1}
		\caption{\footnotesize Construction of $H$ using the $U$ and $V$ variables (in red and blue, respectively). Observe that the second vertex from the left at the bottom is isolated, yet its $U$ variable is equal to $1$, meaning that the infection is able to survive there, which is the reason it is connected to the vertices on the row above.}
\end{center}
\end{figure}
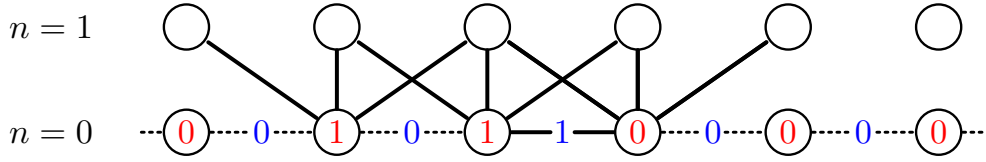

\begin{figure}[h]
	\begin{center}
		
		\begin{tikzpicture}[line cap=round,line join=round,>=triangle 45,x=2.0cm,y=2.0cm]
		\node[scale=1.3] at (-3.4,0) {$n=0$};
		\node[scale=1.3] at (-3.4,0.7) {$n=1$};
		\foreach \i in {-2.5,-1.5,-0.5,0.5,1.5,2.5}
		{
			\draw[line width=0.9pt] (\i,0) circle (0.3cm);
			\draw[line width=0.9pt] (\i,0.7) circle (0.3cm);
		}
		\draw [line width=1.5pt] (-0.34,0.0)-- (0.34,0.0);
		\draw [line width=1.5pt] (-1.635,0.0945) -- (-2.355,0.5985);
		\draw [line width=1.5pt] (-1.365,0.0945) -- (-0.645,0.5985);
		\draw [line width=1.5pt] (-1.5,0.155) -- (-1.5,0.545);
		\draw [line width=1.5pt] (-0.635,0.0945) -- (-1.355,0.5985);
		\draw [line width=1.5pt] (-0.365,0.0945) -- (0.355,0.5985);
		\draw [line width=1.5pt] (-0.5,0.155) -- (-0.5,0.545);
		\draw [line width=1.5pt] (0.365,0.0945) -- (-0.355,0.5985);
		\draw [line width=1.5pt] (0.635,0.0945) -- (1.355,0.5985);
		\draw [line width=1.5pt] (0.5,0.155) -- (0.5,0.545);
		\draw [line width=1.5pt] (0.365,0.0945) -- (-0.355,0.5985);
		\draw [line width=1.5pt] (0.635,0.0945) -- (1.355,0.5985);
		\draw [line width=1.5pt] (0.5,0.155) -- (0.5,0.545);
		\draw[line width=2pt, red] (0.5,0) circle (0.3cm);
		\draw[line width=2pt, red] (-0.5,0) circle (0.3cm);
		\draw[line width=2pt, red] (-1.5,0.7) circle (0.3cm);
		\draw [line width=2pt,red] (-0.635,0.0945) -- (-1.355,0.5985);
		\draw [line width=2pt,red] (-0.34,0.0)-- (0.34,0.0);
		\node[scale=1.3] at (0.5,0) {$a$};
		\node[scale=1.3] at (-0.5,0.0) {$b$};
		\node[scale=1.3] at (-1.5,0.7) {$c$};
		
		\end{tikzpicture}
		\label{consth2}
		\caption{\footnotesize An $H$-valid path from node $a$ to node $c$. The path represents that an infection starting at the block $a$ is able to propagate to block $b$ between times $0$ and $T$, and some of the infected vertices can belong to block $c$ at time $T$.}
	\end{center}
\end{figure}
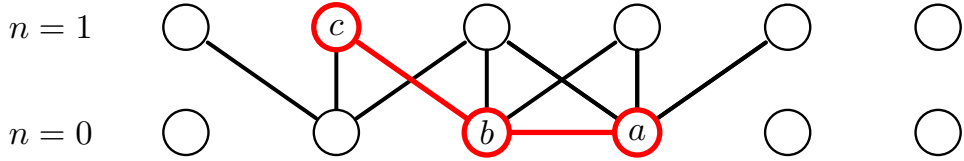

Finally, define a discrete-time process $\{Z_n\}_{n\in\N}$ (which we call $Z(U,V)$ when emphasizing the dependency on the $U$ and $V$ variables) taking values in the family of finite subsets of $\Z$, with $Z_0$ being the set of all $k\in\Z$ such that $B_{k,0}$ contains an initially infected site, and for $n\geq1$, $Z_{n}$ is the set of all the $k\in\Z$ such that there is some $H$-valid path starting at $Z_0\times\{0\}$ and ending in $(k,n)$. The next result shows that $\{Z_n\}_{n\in\N}$ in fact provides a suitable upper bound for $\eta$:

\begin{proposition}
	\label{Zboundseta}
	Take $k\in\Z$ and $n\in\N$. Outside a null probability event, if there is some $x\in\Z$ with $x\in B_{k,n}$ such that $\eta_{nT}(x)=1$, then $k\in Z_n$. In particular, $Z_n=\emptyset$ implies that $\eta_{nT}\equiv 0$ almost surely.
\end{proposition}
\begin{proof}
	If $n=0$, then there is some $x\in B_{k,0}$ such that $\eta_0(x)=1$ and the result follows from the definition of $Z_0$. Now suppose that $n\geq 1$ and take some $x\in B_{k,n}$ such that $\eta_{nT}(x)=1$. Using the graphical construction from Section \ref{graphical} we deduce that there must be a valid path between some $(x_0,0)$ with $\eta_0(x_0)=1$, and $(x,nT)$. Even further, assuming that there are no infection events at times of the form $nT$ (as is almost surely the case), this path defines a unique sequence of space-time points $(x_0,0),(x_1,T),\dotsc,(x_{n-1},(n-1)T),(x_n,nT)$ of points, which in turn defines a sequence of indices $(k_0,0),(k_1,1),\dotsc,(k_n,n)$ representing the blocks containing these points. Noticing that $k_0\in Z_0$, it will be enough to show that there is an $H$-valid path connecting all the $(k_j,j)$'s. Observing that joining two $H$-valid paths results in a $H$-valid path, it will actually be enough to show that there is such a path connecting $(k_0,0)$ and $(k_1,1)$. To do so we consider two cases.
	
	Suppose first that the path joining $(x_0,0)$ and $(x_1,T)$ is entirely contained in $B_{k_0,0}\times[0,T)$. Observe that in this case $x_1$ necessarily belongs to $B_{k_0,0}$ and hence $(k_0-1)r_0<x_1<(k_0+1)r_0$, so in particular $k_0-1\leq k_1\leq k_0+1$. Now, since there is a valid path contained in $B_{k_0,0}\times[0,T)$ (joining points at the top and the bottom of the block), then we necessarily have $U_{k_0,0}=1$ and hence the edge between $(k_0,0)$ and $(k_1,1)$ belongs to $H$, defining an $H$-valid path.
	
	Next suppose that the path joining $(x_0,0)$ and $(x_1,T)$ is not entirely contained in $B_{k_0,0}\times[0,T)$. Let then $k'_0$ be the index of the interval $B_{k'_0,0}$ containing $x_1$, and observe that if $k'_0= k_0$, then we have $k_0-1\leq k_1\leq k_0+1$ as before, and since the path is not contained in $B_{k_0,0}\times[0,T)$ we deduce that either $V_{\{k_0-1,k_0\},0}=1$ or $V_{\{k_0,k_0+1\},0}=1$ so there are edges in $H$ as if $U_{k_0,0}=1$ and we conclude as before. Suppose now that $k'_0> k_0$ (the case $k'_0<k_0$ is analogous) and observe that for any $k_0\leq j< k'_0$ we necessarily have $V_{\{j,j+1\},0}=1$, giving that the edge between $(j,0)$ and $(j+1,0)$ belongs to $H$. As before, we also know that the edge between $(k'_0,0)$ and $(k_1,1)$ belongs to $H$ and hence the path $(k_0,0),(k_0+1,0),\dotsc,(k'_0,0),(k_1,1)$ is $H$-valid, giving the result.
\end{proof}

It follows from its definition that $Z$ is a Markov process obtained as a particular variant of oriented percolation.
We say that this process dies out (or gets extinct) if it ever reaches $\emptyset$ and call $N_\text{ext}$ its extinction time, which, from the proposition above, satisfies $\tau_\text{ext}\leq TN_\text{ext}$ almost surely (where $\tau_\text{ext}$ is the extinction time of the CPDE). It follows that if we can prove a version of Theorem \ref{teo1} for $Z$, that is, $N_\text{ext}<\infty$ a.s. and $\E(N_\text{ext}|Z_0)\leq \beta_0\log(|Z_0|)+\beta_1$, then $\tau_\text{ext}<\infty$ a.s. and
\[\E(\tau_\text{ext}|\eta_0)\,\leq\,\E(TN_\text{ext}|Z_0(\eta_0))\,\leq\,T(\beta_0\log(|Z_0|)+\beta_1)\,\leq\,\tilde{\beta_0}\log(|\eta_0|)+\tilde{\beta}_1,\]
where we have used that $|Z_0|\leq|\eta_0|$ since every $k\in Z_0$ represents an interval $B_{k,0}$ containing an initially infected site.

The construction of $Z$ is relatively simple, but the complexity of the CPDE still remains hidden within the $U$ and $V$ variables, and thus obtaining the result for $Z$ could be just as hard as obtaining it for the original process. What saves us is that, as can be checked directly, $Z(U,V)$ is increasing in the $U$ and $V$ variables, meaning that if for each $k$ and $n$ we have $U_{k,n}\leq U'_{k,n}$ and $V_{\{k,k+1\},n}\leq V'_{\{k,k+1\},n}$, then $Z_m(U,V)\subseteq Z_m(U',V')$ for all $m$. It will be enough then to show that we can find a family of i.i.d. variables $U'$ and $V'$ as above which dominate the $U$ and $V$ variables, and prove a version of Theorem \ref{teo1} for $Z(U',V')$. We begin by showing this latter result:

\begin{lemma}
	\label{zdies}
	Suppose that the $U'_{k,n}$ and $V'_{\{k,k+1\},n}$ are independent Bernoulli random variables with parameter $\epsilon>0$. If $\epsilon$ is small enough, then $Z(U',V')$ dies out a.s. for any finite initial configuration.
	Furthermore, there are $\beta_0,\beta_1>0$ such that
	\[\E(N_\text{ext}|Z_0)\,\leq\,\beta_0\log(|Z_0|)+\beta_1.\]
\end{lemma}

\begin{proof}
	Consider i.i.d. uniform random variables $u'_{k,n}$ and $v'_{\{k,k+1\},n}$ on $[0,1]$ and, for some given $\epsilon>0$, let $U'_{k,n}=1$ and $V'_{\{k,k+1\},n}=1$ if $u'_{k,n}\leq\epsilon$ and $v'_{\{k,k+1\},n}\leq\epsilon$ (this serves the usual purpose of coupling realizations of these variables for different values of $\epsilon$).
	 Fix then some value of $\epsilon$ and define for any finite set $A\subseteq Z$ the function $F_{A}(n)=\P(N_\text{ext}\leq n\,|\,Z_0=A)$. Conditioned on $Z_0=A$, we have $N_\text{ext}\leq n$ if and only if the event
	\[\bigcap_{a\in A}\big\{\text{there are no $H$-valid paths from $(a,0)$ to $\Z\times\{n\}$}\big\}\]
	occurs.
	Each of the events in this intersection is decreasing in the $U'$ and $V'$ variables so we can apply the FKG inequality to obtain
	\begin{align*}
		F_{A}(n)&\geq \prod_{a\in A}\P(\text{there are no $H$-valid paths from $(a,0)$ to $\Z\times\{n\}$})\\
		&=\P(\text{there are no $H$-valid paths from $(0,0)$ to $\Z\times\{n\}$})^{|A|}
		=\big(F_{\{0\}}(n)\big)^{|A|},
	\end{align*}
	where we have used translation invariance.
	Take now $g(x,\epsilon)$ to be the probability generating function of the variable $|Z_1|$ conditioned on $Z_0=\{0\}$, that is,
	\[g(x,\epsilon)\;=\;\sum_{k=0}^\infty \P(|Z_1|=k\,|\,Z_0=\{0\})\,x^{k},\]
	and use the inequality above along with total probabilities and the Markov property to deduce
	\begin{align*}
		F_{\{0\}}(n+1)&=\sum_{\substack{B\subseteq\Z \\ B\text{ finite}}}\P(N_\text{ext}\leq n+1,\,Z_{1}=B\,|\,Z_0=\{0\})\\
		&=\sum_{\substack{B\subseteq\Z \\ B\text{ finite}}}\P(N_\text{ext}\leq n,\,|\,Z_{0}=B)\P(Z_1=B\,|\,Z_0=\{0\})\\
		&\geq \sum_{\substack{B\subseteq\Z \\ B\text{ finite}}}\big(F_{\{0\}}(n)\big)^{|B|}\P(Z_1=B\,|\,Z_0=\{0\})
		=g(F_{\{0\}}(n),\epsilon).
	\end{align*}
	Since $g(\cdot,\epsilon)$ is monotone, we get
	\[F_{\{0\}}(n)\;\geq\;g^{(n)}(F_{\{0\}}(0),\epsilon)\;=\;g^{(n)}(0,\epsilon)\]
	where $g^{(n)}(\cdot,\epsilon)$ is defined inductively as $g^{(n+1)}(x,\epsilon)=g(g^{(n)}(x,\epsilon),\epsilon)$. Now for any fixed $\epsilon$ we have that for all $x\in(0,1)$,
	\[\frac{1-g(x,\epsilon)}{1-x}\;=\;\sum_{k=1}^{\infty}\P(|Z_1|=k\,|\,Z_0=\{0\})\frac{1-x^k}{1-x}\;\leq\;\E(|Z_1|\,\big|\,Z_0=\{0\}),\]
	so that $1-g^{(n+1)}(0,\epsilon)\leq \E(|Z_1|\,\big|\,Z_0=\{0\})[1-g^{(n)}(0,\epsilon)]$ and hence we deduce inductively that
	\[1-F_{\{0\}}(n)\,\leq\,1-g^{(n)}(0,\epsilon)\,\leq\,\E(|Z_1|\,\big|\,Z_0=\{0\})^n.\]
	From the construction of the $U'$ and $V'$ variables and the fact that $Z$ is increasing in them, it follows that $|Z_1|$ decreases to zero a.s. as $\epsilon\to 0$, so if $\E(|Z_1|\,\big|\,Z_0=\{0\})<\infty$ monotone convergence shows that for small enough $\epsilon$
	\begin{equation*}\label{eq:Ebd}
	\gamma:=\E(|Z_1|\,\big|\,Z_0=\{0\})<1,
	\end{equation*}
	from which we conclude $F_{\{0\}}(n)\longrightarrow1$.
	We deduce that for any finite $B\subseteq\Z$
	\[\P(N_\text{ext}=\infty\,|\,Z_{0}=B)\,=\,\lim_{n\to\infty}(1-F_{B}(n))\,\leq\,\lim_{n\to\infty}(1-F_{\{0\}}(n)^{|B|})\,=\,0,\]
	and hence that $Z$ dies out almost surely.
	Furthermore, for such small $\epsilon$ an easy calculation gives
	\begin{align*}
		\E(N_\text{ext}\,|\,Z_0=B)&=\sum_{k=1}^\infty\big(1-F_{B}(k)\big)\;\leq\;\sum_{k=1}^\infty\big(1-(1-\gamma^k)^{|B|}\big)\\
		&\leq\;-\log_{\gamma}(|B|)+1\,+\,\sum_{k=\lceil-\log_{\gamma}(|B|)\rceil}^\infty\big(1-(1-\gamma^k)^{|B|}\big),
	\end{align*}
	and if $|B|$ is large, using the inequality $1-x\geq e^{-2x}$ for small $x$ the sum on the right can be bounded by $\sum_{k=\lceil-\log_{\gamma}(|B|)\rceil}^\infty 2|B|\gamma^k\;\leq\;\frac{2|B|}{1-\gamma}\gamma^{-\log_{\gamma}(|B|)}\,=\,\frac{2}{1-\gamma}$,
	giving the required bound on the expectation.

	It only remains to prove $\E(|Z_1|\,\big|\,Z_0=\{0\})<\infty$. To this end observe that given $Z_0=\{0\}$, $Z_1$ is either empty or an interval containing $0$, and in the latter case, this interval is given as $[e^-_l,e^+_r]$ where $e_l$ and $e_r$ are the first edges to the left and right of $0$, respectively, such that $V'_{e,0}=0$. Fixing $\epsilon$ as before, it follows that for each $k\geq 4$ there are $k-2$ possible such intervals of length $k$, each one occurring with probability $\epsilon^{k-3}(1-\epsilon)^2$, and hence the sum $\sum_{k=0}^\infty k\P(|Z_1|=k\,|\,Z_0=\{0\})$ is finite. 
\end{proof}

Thanks to the lemma, all that remains to prove is that for any $\lambda>0$, $p\in(0,1)$ and $\epsilon>0$, we can choose $v$, $T$ and $r_0$ in such a way that we can couple our $U$ and $V$ variables with an i.i.d. Ber$(\epsilon)$ family. To this end take $v=1/T$ and fix $\delta_0=\frac{e^{-1}+(1-p)(1-e^{-1})-e^{-p}}{1-e^{-p}}e^{-p}\in(0,1)$, which, as we will show later with the aid of \eqref{ineqv}, is a lower bound for the probability of an edge being $n$-closed for this choice of $v$. Next, take $r_0\in\N$ large enough so that $(1-\delta_0)^{r_0}<\epsilon$, and choose $T>0$ to be sufficiently large so that, letting $\tau_\text{ext}$ be the extinction time of the standatd contact process running on the interval $[0,2r_0]$ starting with every vertex infected, we have $\P\big(\tau_\text{ext}\geq T\big)<\epsilon$.

Now, observe that all the $V$ variables depend only on $\zeta$, and that given a realization of the environment the $U$ variables depend on the $\mI$ and $\mR$ processes on disjoint sets, so they are independent. Furthermore, each block $B_{k,n}$ has length at most $2r_0-1$ so our choice of $T$ yields $\P(U_{k,n}=1\;|\;\zeta)<\epsilon$ as desired for any realization of $\zeta$. 
On the other hand, each $V_{\{k,k+1\},n}$ depends on the $\mU$ process and the $\mC$ variables on $B_{k,n}$, so $V_{e,n}$ and $V_{e',m}$ are independent as long as $e\neq e'$. So all that remains to prove now is that
\begin{equation*}
\label{desiV}
\P\big(V_{e,n}=1\;\big|\;V_{e,n-1},\cdots,V_{e,0}\big)\leq\epsilon
\end{equation*}
for any $e\in E$ and $n\in\N$. This inequality will follow from the next result, whose easy but tedious proof is deferred until the end of this section.

\begin{proposition}
	\label{propdesiV}
	Fix $v,p$ and $T$. For given $e\in E$ define $w^e_n={\bf 1}_{\text{$\zeta_t(e)=1$ for some $t\in[nT,(n+1)T)$}}$. Then for all $n\geq 1$,
	\begin{equation}
	\label{ineqv}
	\P(w^e_n=0\,|\,w^e_{n-1},\dotsc,w^e_0)\geq\delta\coloneqq e^{-pvT}\!\left[\tfrac{e^{-vT}+(1-p)(1-e^{-vT})-e^{-pvT}}{1-e^{-pvT}}\right],
	\end{equation}
	and for $n=0$ we have $\P(w^e_0)\geq\delta$.
\end{proposition}

\noindent Notice that replacing $v=1/T$ in the definition of $\delta$ we recover $\delta_0$ and using that
$$V_{\{k,k+1\},n}=1\;\Longleftrightarrow\;\mbox{for all edges in }[kr_0,(k+1)r_0]\mbox{ we have }w_n^e=1,$$
Proposition \ref{propdesiV}, together with the independence of the $\zeta_t$ processes on different edges, imply that
$$\P\big(V_{e,m}=1\;\big|\;V_{e,m-1},\cdots,V_{e,0}\big)\;\leq\;(1-\delta_0)^{r_0}\leq \epsilon.$$
This completes the proof of the theorem.

\subsection{Proof of Theorem \ref{teo2}}

The proof of this theorem follows closely what was done in the proof of Theorem \ref{teo1}, where for fixed $p$ and $\lambda$ we were able to couple the CPDE running on $\Z$ with a process $Z$ which dies out provided that $v$ was sufficiently small. Surprisingly enough, when taking $v$ fixed and $p$ small instead, it is possible to adapt the proof to obtain not only that the result holds for general $\G$, but also that the coupled process $Z$ dies out {\it independently of the value of $\lambda$}. Indeed, we will show that if $p$ is sufficiently small, then we can divide time into intervals of the form $[nT,(n+1)T)$ where every edge has a very large probability of being $n$-closed (as opposed to our previous proof, where we could only find some $n$-closed edge within a large enough interval $B_{k,n}$). In particular, most vertices will be isolated throughout these intervals and as a result, infections confined to these quarantined vertices die out in time of order 1.

Fix $\lambda>0$, $v>0$, and take $\varepsilon>0$ small (to be fixed later), which we use to define an auxiliary parameter $M=2\max\{\varepsilon^{-1},v\log(\varepsilon^{-1})\}$.
Next choose $p$ small enough so that
\begin{equation}
\label{ineq2.2}
e^{-pM}\left[\tfrac{e^{-M}+(1-p)(1-e^{-M})-e^{-pM}}{1-e^{-pM}}\right]\geq1-\varepsilon
\end{equation}
(which is possible by our choice of $M$ because the expression on the left goes to $1-\frac{1-e^{-M}}{M}\geq1-\frac{1}{M}$ as $p\rightarrow0$). Next, partition $\R^+$ into time intervals of the form $[nT,(n+1)T)$, where $T=M/v$, and define the variables $\{V_{e,n}\}$ and $\{U_{x,n}\}$ for $x\in G$, $e\in E$ and $n\in\N$ as
\begin{align*}
	V_{e,n}&={\bf 1}_{\zeta_t(e)=1\text{ for some }\;t\in[nT,(n+1)T)},\\
	U_{x,n}&={\bf 1}_{\mathcal{R}^{x}\cap[nT,(n+1)T)=\emptyset}.
\end{align*}

Observe that if $V_{e,n}=0$ then the infection cannot use $e$ to spread throughout $[nT,(n+1)T)$ and hence if $\sum_{y\sim x}V_{\{x,y\},n}=0$ we deduce that $x$ is isolated during this time period. We will show that this is very likely to happen, and in that scenario, $U_{x,n}=0$ implies that any infection at $x$ dies out before time $(n+1)T$. Using these variables, which do not depend on the infection processes, we construct a graph $H$ with vertex set $G\times\N$ and whose edge set is obtained from the $U$ and $V$ variables according to the following rules:
\begin{enumerate}
	\item If $U_{x,n}=1$, then add the edge between $(x,n)$ and $(x,n+1)$.
	\item For $e=\{x,y\}$, if $V_{e,n}=1$, then add edges as if $U_{x,n}=1$ and $U_{y,n}=1$, and also the edge between $(x,n)$ and $(y,n)$.
\end{enumerate}
The resulting graph is similar to the one defined in the proof of Theorem \ref{teo1}, and in terms of it we define $H$-valid paths and the process $Z$ analogously to what was done there. It can be checked that for this new process $Z$ analogous versions of Proposition \ref{Zboundseta} and Lemma \ref{zdies} hold, so in order to conclude both statements of the theorem, it is enough to bound the $U$ and $V$ variables by an i.i.d. family of Bernoulli random variables with small enough parameter $\varepsilon$. Notice that the $U$ variables are independent from one another, and are also independent from the $V$ variables, which in turn are also independent among themselves whenever indexed by different edges.  As a result, all we need to show is that
$$\P\big(U_{x,n}=1\big)\leq\varepsilon\quad\text{ and }\quad\P\big(V_{e,n}=1\;\big|\;V_{e,n-1},\cdots,V_{e,0}\big)\leq\varepsilon$$ for any $x\in G$, $e\in E$ and $n\in\N$. For the first inequality, observe that $\P\big(U_{k,n}=1\big)= e^{-T}=e^{-M/v}\leq\varepsilon$ from our choice of $M$, while for the second inequality we can use Proposition \ref{propdesiV} directly with $w_n^e=V_{e,n}$ to obtain
$$\P\big(V_{e,n}=1\;\big|\;V_{e,n-1},\cdots,V_{e,0}\big)\leq 1-\delta$$
with $\delta$ as defined in \eqref{ineqv}. Replacing $T=M/v$ we see that $\delta$ is equal to the expression on the left hand side of \eqref{ineq2.2}, and the desired inequality follows. We conclude that $Z$ (and hence the CPDE) dies out, and the result then follows from noticing that the construction does not depend on $\lambda$, so that $\lambda_0(v,p)=\infty$.
 
 \subsection{Proof of Theorem \ref{teo3}}
 
 Our goal is to prove that if $p$ is sufficiently close to $1$, then for every $v>0$ we can take $\lambda$ large enough so that the infection process $\eta$ survives. To this end we use again a block construction argument, this time based on the usual comparison with oriented percolation as introduced in \cite{orientedperc}.
 As mentioned in Section \ref{sec:genG} it is enough to prove this result in the case $\G=\Z$, so we make this assumption.
 
 \hskip-12pt
 \begin{minipage}{0.7\textwidth}
 	\mbox{\hspace*{7.5pt}} For any $T>0$ we divide $\Z\times[0,\infty)$ into blocks $B_{k,n}$ the form
 	\smash{$B_{k,n}=\mathfrak{I}_{k,n}\times[nT,(n+1)T)$} with \smash{$\mathfrak{I}_{k,n}=\big\{4k-2n,\dotsc,4k-2n+3\big\}.$}
 	Note that, with this choice, half of each block lies on top of each of the two adjacent blocks in the row below it, as shown in the picture. For each $k$ and $n$ we say that the block $B_{k,n}$ is ``good", an event which we denote as $\mathcal{W}_{k,n}$, if the following conditions hold:
 \end{minipage}%
 \hskip10pt
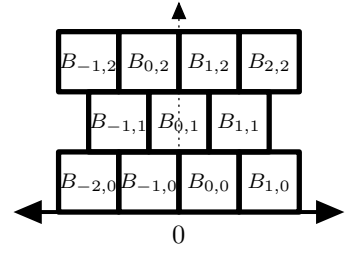
 \begin{minipage}{0.27\textwidth}
 	\vskip-10pt
 	\begin{center}
 		\begin{tikzpicture}[line cap=round,line join=round,>=triangle 45,x=1.0cm,y=1.0cm]
 		\draw [line width=1.5pt,<->] (-2.2,0.0)-- (2.2,0.0);
 		\draw [line width=0.5pt,->,dotted] (0.0,0.0)node[xshift=-0.0cm,yshift=-0.3cm]{$0$}-- (0.0,2.8);
 		\foreach \a in {-2, -1, ..., 1 }
 		\draw [line width=2pt] (\a*0.8,0.0)node[xshift=0.4cm,yshift=0.375cm]{{\footnotesize $B_{\a,0}$}} rectangle (\a*0.8+0.8,0.8);
 		\foreach \a in {-1, 0, ..., 1 }
 		\draw [line width=2pt] (\a*0.8-0.4,0.8)node[xshift=0.4cm,yshift=0.375cm]{{\footnotesize $B_{\a,1}$}} rectangle (\a*0.8+0.4,1.6);
 		\foreach \a in {-1, 0, ..., 2 }
 		\draw [line width=2pt] (\a*0.8-0.8,1.6)node[xshift=0.4cm,yshift=0.375cm]{{\footnotesize $B_{\a,2}$}} rectangle (\a*0.8,2.4);
 		\end{tikzpicture}
 	\end{center}
 \end{minipage}
 
 \begin{enumerate}[label=\normalfont{(c\arabic*)}]
 	\item \label{c1} For each edge $e$ lying inside $\mathfrak{I}_{k,n}$ we have $\mathcal{O}^{e}\cap[nT,(n+1)T)\neq\emptyset$.
 	\item \label{c2} For each edge $e$ lying inside $\mathfrak{I}_{k,n}$ we have $\mathcal{C}^{e}\cap[nT,(n+1)T)=\emptyset$.
 	\item \label{c3} Let $T_{k,n}=\bigcup_{x,e\text{ inside }\mathfrak{I}_{k,n}}(\mathcal{R}^{x}\cup\mathcal{C}^{e}\cup\mathcal{O}^{e})\cap[nT,(n+1)T)$. Then $|t_1-t_2|>\delta$ for all $t_1,t_2\in T_{k,n}\cup\{nT,(n+1)T\}$.
 	\item \label{c4} For all edge $e$ lying inside $\mathfrak{I}_{k,n}$ we have $\mathcal{I}^e\cap\left[nT+\frac{lT\delta}{6},\;nT+\frac{(l+1)T\delta}{6}\right]\neq\emptyset$ for all $0\leq l<\frac{6}{\delta}$.
 \end{enumerate}

 In words, conditions \ref{c1} and \ref{c2} say that all edges lying inside $\mathfrak{I}_{k,n}$ become available at some time in $[nT,(n+1)T)$ and remain so until time $(n+1)T$, while conditions \ref{c3} and \ref{c4}, on the other hand, ensure that between two non-infection events there is a (not necessarily valid) infection between each pair of neighbouring vertices in $\mathfrak{I}_{k,n}$. 
 
 Next we choose the parameters of our model and our block construction. Fix $\epsilon>0$ and take $M>0$ large enough so that $e^{-M/2}\leq\epsilon/12$ and then $p$ close enough to $1$ so that $1-e^{-M(1-p)}\leq\epsilon/12$ (this will determine the value of the parameter $p_1$ which we are looking for in Theorem \ref{teo3}). Now fix $v>0$ and take $T=M/v$. Focusing on a single block $B_{k,n}$ we have that: $\P\text{(c1)}=(1-e^{-Mp})^3>1-\frac{\epsilon}{4}$ from our choice of $M$; $\P\text{(c2)}=e^{-3(1-p)M}>1-\frac{\epsilon}{4}$ from our choice of $p$; $\P\text{(c3)}\nearrow1$ as $\delta\searrow0$, so we can choose $\delta>0$ small enough such that $\P\text{(c3)}\geq1-\epsilon/4$; and having fixed $\delta$, $\P\text{(c4)}\nearrow1$ as $\lambda\nearrow\infty$, so we can take $\lambda$ large enough such that $\P\text{(c4)}\geq1-\epsilon/4$.
 From this choice of the parameters we conclude that $\P(\mathcal{W}_{k,n})\geq 1-\varepsilon$ for each $k$ and $n$.
 
 Take now the directed graph with vertex set $\{B_{k,n}\}_{k\in\Z,n\in\N}$ and where each block $B_{k,n}$ has $B_{k,n+1}$ and $B_{k+1,n+1}$ as directed neighbors. Since all of the events $\mathcal{W}_{k,n}$ are independent, by taking the subgraph of all the blocks satisfying these events we recover the two-dimensional site percolation model of \cite{durrett1980} with percolation parameter at least $1-\varepsilon$. By choosing $\varepsilon$ sufficiently small we deduce that with positive probability there is an infinite path $(B_{k_n,n})_{n\in\N}$ starting at $B_{0,0}$ (which, from our construction of the network satisfies $k_{n+1}\in\{k_n,k_n+1\}$ for each $n$). Observe that by choosing $\varepsilon$ we fix the value of $p=p_1$, while the parameter $v$ only determines $T$. We claim that on the event where this infinite path exists, we obtain survival of $\eta$ as soon as $\mathfrak{I}_{0,0}$ contains two adjacent vertices $x_0,y_0$ such that $\eta_{0}(x_0)=\eta_0(y_0)=\zeta_0(\{x_0,y_0\})=1$.
 
 To prove this claim, observe that conditions \ref{c1}-\ref{c4} imply that at time $T$, $\eta_T(x)=1$ and $\zeta_T(e)=1$ for each vertex $x$ and edge $e$ in $\mathfrak{I}_{0,0}$. Indeed, from condition \ref{c1}, if an edge is absent in $\mathfrak{I}_{0,0}$ at time $0$, then it appears at some point in the time interval $[0,T)$. On the other hand, from condition \ref{c2} no edge can disappear in this interval.
 We deduce that the edge $\{x_0,y_0\}$ remains available throughout $[0,T)$ and that all edges are available at time $T$, giving $\zeta_{T}\equiv 1$ inside $\mathfrak{I}_{0,0}$.
 
 \hskip-12pt
 \begin{minipage}{0.7\textwidth}
 	\vskip4pt
 	\hskip12pt Furthermore, observe that from condition \ref{c3} we can actually deduce that $\zeta_{T-\delta}\equiv1$ inside $\mathfrak{I}_{0,0}$, since there are no updating or recovery events in $[T-\delta,T]$.
 	To deduce the analogous result for $\eta$ enumerate the recovery events of $x_0$ and $y_0$ as $r_1,r_2,\dotsc$.
 	Conditions \ref{c3} and \ref{c4} imply that there is always an infection event between these vertices at each interval $(r_{j},r_{j+1})$ which is valid since the edge $\{x,y\}$ is available at all times. In particular, we deduce that at time $T-\delta$ either $x$ or $y$ (or both) are infected, but in the time interval $[T-\delta,T)$ there are no infection or recovery events and all edges are available, so from condition \ref{c4} we can easily obtain the existence of valid infection paths from $x$ and $y$ to all sites in $\mathfrak{I}_{0,0}$.
 \end{minipage}\hskip0.6in
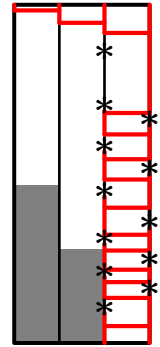
 \begin{minipage}{0.15\textwidth}
 	\begin{center}
 		\vskip0.04in
 		\begin{tikzpicture}[line cap=round,line join=round,>=triangle 45,x=1.2cm,y=1.5cm]
 		\fill [gray] (-0.75,0.0) rectangle (-0.25,1.4);
 		\fill [gray] (-0.25,0.0) rectangle (0.25,0.83);
 		\draw [line width=1.5pt] (-0.75,0.0) rectangle (0.75,3.0);
 		\draw [line width=1pt] (-0.25,0.0) -- (-0.25,3.0);
 		\draw [line width=1pt] (0.25,0.0) -- (0.25,3.0);
 		\draw [red, line width=1.6pt] (0.25,0.0) -- (0.25,0.27);
 		\draw [red, line width=1.6pt] (0.75,0.0) -- (0.75,0.44);
 		\draw [red, line width=1.6pt] (0.25,0.15) -- (0.75,0.15);
 		\draw [red, line width=1.6pt] (0.25,0.4) -- (0.75,0.4);
 		\draw [red, line width=1.6pt] (0.25,0.4) -- (0.25,0.59);
 		\draw [red, line width=1.6pt] (0.25,0.54) -- (0.75,0.54);
 		\draw [red, line width=1.6pt] (0.75,0.54) -- (0.75,0.69);
 		\draw [red, line width=1.6pt] (0.25,0.65) -- (0.75,0.65);
 		\draw [red, line width=1.6pt] (0.25,0.65) -- (0.25,0.88);
 		\draw [red, line width=1.6pt] (0.25,0.82) -- (0.75,0.82);
 		\draw [red, line width=1.6pt] (0.75,0.82) -- (0.75,1.02);
 		\draw [red, line width=1.6pt] (0.25,0.96) -- (0.75,0.96);
 		\draw [red, line width=1.6pt] (0.25,0.96) -- (0.25,1.3);
 		\draw [red, line width=1.6pt] (0.25,1.2) -- (0.75,1.2);
 		\draw [red, line width=1.6pt] (0.75,1.2) -- (0.75,1.5);
 		\draw [red, line width=1.6pt] (0.25,1.45) -- (0.75,1.45);
 		\draw [red, line width=1.6pt] (0.25,1.45) -- (0.25,1.7);
 		\draw [red, line width=1.6pt] (0.25,1.63) -- (0.75,1.63);
 		\draw [red, line width=1.6pt] (0.75,1.63) -- (0.75,1.94);
 		\draw [red, line width=1.6pt] (0.25,1.85) -- (0.75,1.85);
 		\draw [red, line width=1.6pt] (0.25,1.85) -- (0.25,2.06);
 		\draw [red, line width=1.6pt] (0.25,2.04) -- (0.75,2.04);
 		\draw [red, line width=1.6pt] (0.75,2.04) -- (0.75,3);
 		\draw [red, line width=1.6pt] (0.25,2.75) -- (0.75,2.75);
 		\draw [red, line width=1.6pt] (0.25,2.75) -- (0.25,3);
 		\draw [red, line width=1.6pt] (-0.25,2.84) -- (0.25,2.84);
 		\draw [red, line width=1.6pt] (-0.25,2.84) -- (-0.25,3);
 		\draw [red, line width=1.6pt] (-0.75,2.95) -- (-0.25,2.95);
 		\draw [red, line width=1.6pt] (-0.75,2.95) -- (-0.75,3);
 		\node at (0.25,0.27) [black,fontscale=2,thick]{\bf *};
 		\node at (0.25,0.59) [black,fontscale=2,thick]{\bf *};
 		\node at (0.25,0.88) [black,fontscale=2,thick]{\bf *};
 		\node at (0.25,1.3) [black,fontscale=2,thick]{\bf *};
 		\node at (0.25,1.7) [black,fontscale=2,thick]{\bf *};
 		\node at (0.25,2.06) [black,fontscale=2,thick]{\bf *};
 		\node at (0.25,2.54) [black,fontscale=2,thick]{\bf *};
 		\node at (0.75,0.44) [black,fontscale=2,thick]{\bf *};
 		\node at (0.75,0.69) [black,fontscale=2,thick]{\bf *};
 		\node at (0.75,1.02) [black,fontscale=2,thick]{\bf *};
 		\node at (0.75,1.5) [black,fontscale=2,thick]{\bf *};
 		\node at (0.75,1.94) [black,fontscale=2,thick]{\bf *};
 		\end{tikzpicture}
 	\end{center}
 \end{minipage}
 
 Now half of the block $B_{k_1,1}$ lies on top of $B_{0,0}$, so from the observation, $\mathfrak{I}_{k_1,1}$ contains two adjacent vertices $x_1,y_1$ such that $\eta_{T}(x_1)=\eta_T(y_1)=\zeta_T(\{x_1,y_1\})=1$, and we can repeat the argument above to conclude that at time $2T$,  $\eta_{2T}(x)=1$ and $\zeta_{2T}(e)=1$ for each vertex $x$ and edge $e$ in $\mathfrak{I}_{k_1,1}$. Repeating this argument iteratively we conclude that at each time $nT$ there are vertices $x,y\in\mathfrak{I}_{k_n,n}$ such that $\eta_{nT}(x)=\eta_{nT}(y)=1$, yielding survival of $\eta$.

 \subsection{Proof of Proposition \ref{propdesiV}}
 Recall that $w_n^e=0$ is equivalent to $\zeta_t(e)=0$ for all $t\in[nT,(n+1)T)$ and notice that $\P\big(w_n^e=0\big|\mathcal{F}_{nT}\big)=e^{-pvT}\mathds{1}_{\{\zeta_{nT}(e)=0\}}$, where $\mathcal{F}_{nT}$ is the $\sigma$-algebra generated by $\zeta$ up until time $nT$. Thus it is enough to show the improved inequality
 \begin{equation*}
 \label{nuevo}\P(\zeta_{nT}(e)=0|w_{n-1}^e,\zeta_{(n-1)T}(e),w_{n-2}^e,\dotsc,w_0^e)\;\geq\;\tfrac{e^{-vT}+(1-p)(1-e^{-vT})-e^{-vpT}}{1-e^{-vpT}}=e^{pvT}\delta\eqqcolon\delta'
 \end{equation*}
for $n\geq 1$ and $\P(\zeta_{0}(e)=0)=(1-p)\geq\delta'$ for $n=0$. Observing that the probability above is equal to $\P(\zeta_{T}(e)=0|w_{0}^e,\zeta_{0}(e))$ due to the Markov property and homogeneity of the environment, there are three cases to be considered when $n\geq1$.
 
Assume first that $\zeta_{0}(e)=1$, which by definition implies $w_0^e=1$. Then $\zeta_{T}(e)=0$ if and only if $\mathcal{U}\cap[0,T)\neq\emptyset$ and the last updating event is in $\mathcal{C}^e$. The probability of such an event is $(1-p)(1-e^{-vT})$, and hence we need to show that this expression is larger than $\delta'$. To do so, observe that after multiplying by $e^{vpT}-1$ and rearranging terms, the inequality $(1-p)(1-e^{-vT})\geq\delta'$ is equivalent to
 \[e^{-(1-p)vT}+(1-p)(1-e^{-vT})-1\,\leq\,0,\]
 which always holds, since the function $x\longmapsto e^{-xvT}+x(1-e^{-vT})-1$ is convex and equal to zero at $x=0$ and $x=1$. Observe that this bound also gives the result for $n=0$, since $\P(\zeta_0(e)=0)\geq(1-p)(1-e^{-vT})\geq\delta'$.
 
 Assume now that $\zeta_0(e)=0$ and $w_0^e=1$. Here we compute the conditional probability directly. For the numerator, it is easy to see that the event $\{\zeta_{T}(e)=0,\;w_0^e=1,\;\zeta_{0}(e)=0\}$ corresponds to $\zeta_{0}(e)=0$ and $|\mathcal{U}^e\cap[0,T)|\geq 2$ with the last updating event belonging to $\mathcal{C}^e$ and from the rest at least one belonging to $\mathcal{O}^e$. The probability of this event is $(1-p)^2\sum_{n=2}^\infty\left[1-(1-p)^{n-1}\right]e^{-vT}\frac{(vT)^n}{n!}=(1-p)[e^{-vT}+(1-p)(1-e^{-vT})-e^{-vpT}]$. For the denominator, the event $\{w_0^e=1,\;\zeta_{0}(e)=0\}$ corresponds to $\zeta_{0}(e)=0$ and $\mathcal{O}^e\cap[0,T)\neq\emptyset$, so it has probability $(1-p)(1-e^{-vpT})$. Dividing the two expressions we obtain $\delta'$.
 
 Finally, for the case $\zeta_0(e)=0$ and $w_0^e=0$, we have by definition of $w_0^e$ that $\zeta_{T}(e)=0$, so the conditional probability is equal to $1$, and the result follows.

\bigskip

\noindent{\bf Acknowledgements.} 
\enspace The authors thank two anonymous referees for their careful review and several suggestions on the first draft of this paper.
Both authors were supported by Programa Iniciativa Cient\'ifica Milenio grant number NC120062 through Nucleus Millenium Stochastic Models of Complex and Disordered Systems, and by Conicyt Basal-CMM Proyecto/Grant PAI AFB-170001. DR was also supported by Fondecyt Grant 1160174, while AL was supported by the CONICYT-PCHA/Doctorado nacional/2014-21141160 scholarship.

\smallskip

\providecommand{\bysame}{\leavevmode\hbox to3em{\hrulefill}\thinspace}
\providecommand{\MR}{\relax\ifhmode\unskip\space\fi MR }
\providecommand{\MRhref}[2]{%
  \href{http://www.ams.org/mathscinet-getitem?mr=#1}{#2}
}
\providecommand{\href}[2]{#2}



\end{document}